\documentclass{amsart}
\usepackage{latexsym,amsmath,amsfonts,amsthm,amssymb,color}
\usepackage{mathrsfs}
\usepackage[all]{xy}
\usepackage{enumerate}
\usepackage{soul}
\usepackage{hyperref}
\usepackage{mathabx}
 \hypersetup{
     unicode=false,          		
     pdftoolbar=true,        		
     pdfmenubar=true,        		
     pdffitwindow=false,     		
     pdfstartview={FitH},    		
    pdfnewwindow=true,      		
    colorlinks=true,       			
     linkcolor=blue,          		
     citecolor=green,    		
    filecolor=magenta,      		
     urlcolor=cyan }          

\usepackage[usenames,dvipsnames]{xcolor}
\usepackage[shortlabels]{enumitem}
\usepackage[normalem]{ulem}

\usepackage{todonotes}

\newtheorem{theorem}{Theorem}
\newtheorem{lemma}[theorem]{Lemma}
\newtheorem{proposition}[theorem]{Proposition}
\newtheorem{corollary}[theorem]{Corollary}

\newtheorem{remark}[theorem]{Remark}

\numberwithin{equation}{section}
\numberwithin{theorem}{section}

\newcommand{\R}{{\mathbb R }}

\newcommand{\C}{{\mathbb C }}
\newcommand{\T}{\mathbb{T}}
\newcommand{\Z}{{\mathbb Z }}

\newcommand{\dbar}{\bar \partial}

\newcommand{\qbar}{\bar q}
\newcommand{\rbar}{\bar r}
\newcommand{\ubar}{\bar u}

\newcommand{\dnuz}{dz}

\newcommand{\ol}{\overline}

\newcommand{\scrT}{\mathscr{T}}

\newcommand{\calB}{\mathcal{B}}
\newcommand{\calF}{\mathcal{F}}

\newcommand{\calM}{\mathcal{M}}

\newcommand{\calS}{\mathcal{S}}

\newcommand{\dotarg}{\, \cdot	\, }

\newcommand{\diff}{\partial}

\newcommand{\bu}{{\bar u}}

\newcommand{\norm}[2][ ]{\left\Vert #2 \right\Vert_{#1}}

\newcommand{\DX}{D\dot{X}}
\usepackage{mathrsfs}

\begin{document}
	\title[Large data global well-posedness for the mNV system]{Large data global well-posedness for the modified Novikov-Veselov system}
	\author{Adrian Nachman}
	\address{Department of Mathematics, University of Toronto, Toronto ON MS5-2E4, Canada}
    \email{nachman@math.toronto.edu}
	\author{Peter Perry}
	\address{Department of Mathematics, University of Kentucky, Lexington KY 40506-0027, USA}
    \email{pperr0@uky.edu}
	\author{Daniel Tataru}
	\address{Department of Mathematics, University of California, Berkeley, CA 94720, USA}
    \email{tataru@math.berkeley.edu}

\begin{abstract}
The modified Novikov-Veselov system (mNV) is a cubic  third order dispersive evolution in two space dimensions. It is also completely integrable, belonging to the same hierarchy as the defocusing Davey-Stewartson~II (DS II) system. 

The mNV system is $L^2$ critical. Some time ago, Schottdorf proved that for small $L^2$ initial data, the mNV equation is globally well-posed. In this article, we consider instead the large data problem, using inverse scattering methods. Our main result asserts that
the mNV system is globally well-posed for large
$L^2$ data, with the solutions scattering as time goes to $\pm \infty$.  One key
ingredient in the proof, which is of independent interest, is a new nonlinear Gagliardo-Nirenberg inequality  for the associated scattering
transform.

As a byproduct of our main result, we are also able to prove a global well-posedness result for the closely related Novikov-Veselov problem at the critical $\dot H^{-1}+L^1$ level, for a range of data which can heuristically be described as soliton-free. Here we use the associated Miura map to connect the mNV and the NV flows. 
In order to characterize the range of the Miura map, we prove another result of independent interest, namely a sharp, scale invariant form of the  Agmon–Allegretto–Piepenbrink principle in the critical case of two space dimensions.

\end{abstract}


\maketitle

\setcounter{tocdepth}{1}    
\tableofcontents

\section{Introduction}
The Novikov-Veselov equation (NV) and the modified Novikov-Veselov equation (mNV) are third order dispersive flows in two space dimensions,
with quadratic, respectively cubic nonlinearities. They are both completely integrable, and one may think of them as two-dimensional counterparts of the KdV, respectively the mKdV flow. In particular, 
there is a Miura map connecting mNV and NV,
allowing one to transfer results between the two flows.

Our primary interest here is in the mNV equation, which belongs to the same integrable hierarchy as the Davey-Stewartson II (DS II) flow, which is second order, still with cubic 
nonlinearity. This is exactly akin to the connection between the mKdV flow and the cubic NLS flow in one space dimension.
In particular both mNV and DS II are $L^2$ critical, so a very natural question is to study their $L^2$ well-posedness. In this context, there is a fundamental difference between the small and the large data problem.

For the DS II problem, small $L^2$ data global well-posedness was proved early on by Ghidaglia-Saut~\cite{GS1990}, using perturbative methods. However, the large data result is highly nonperturbative, and large $L^2$ data global well-posedness was only proved recently in \cite{NRT2020}. Turning our attention to our primary objective here, namely the mNV flow, small $L^2$ data global well-posedness turned out to be more complex 
than in the DS II case, and was proved by Schottdorff~\cite{Schottdorf13}, still using perturbative methods but with respect to more
complex function spaces and estimates. But the large $L^2$ data problem
has remained open until present.

This is the main question we address in this work, where we prove that global well-posedness holds for mNV for large $L^2$ data. Further, our approach also provides a global description of the solutions, as follows (see Theorem~\ref{thm:mNV} for the precise statements):
\begin{itemize}
    \item the solutions satisfy global dispersive bounds, depending only on the $L^2$ size of the data.
\item the solutions satisfy pointwise bounds, depending on the  solutions of a corresponding linear flow. 
    \item the solutions scatter classically at infinity; moreover, the wave operators are 
    global smooth diffeomorphisms of $L^2$.
\end{itemize}

Once we have the sought after mNV result, in the last section of the paper we also explore the consequences for the NV problem, using the 
corresponding Miura map. This yields a global result for initial data in a specified subset of the space $(\dot H^{-1} + L^1)(\R^2)$, which we are able to fully characterize in terms of the spectral properties of a Schr\"{o}dinger operator, which is  the Lax  operator associated to the NV flow.

\subsection{ The NV and mNV flows}
 We denote the coordinates in $\R^2$ by $(x,y)$, and the 
corresponding complex differentiation operators 
are given by 
\[
\diff=\frac12(\diff_x - i \diff_y),  \qquad \dbar = \frac12(\diff_x+i \diff_y).
\]
Then the Cauchy problem for NV 
for a real valued function $q$ in $\R \times \R^2$
is
\begin{equation}
    \label{NV}
\left\{
\begin{aligned}
   & q_t + (\partial^3 + \dbar^3) q =  N_{NV}(q),\\
    \nonumber
    & \left. q \right|_{t=0} = q_0,
 \end{aligned}
\right.
\end{equation}   
where the quadratic nonlinearity is given by
\begin{equation}
    \label{NNV}
    \frac43 N_{NV}(q) = \diff \left( q \dbar^{-1} \diff q  \right) + \dbar \left( \qbar \diff^{-1} \dbar \qbar \right).
\end{equation}
The NV equation is one of a hierarchy of completely integrable, dispersive equations discovered by Novikov and Veselov \cite{NV1986,NV1984}: the scattering transform which linearizes the NV flow is determined by the two-dimensional Schr\"{o}dinger operator $-\Delta + q$ at fixed energy $E$, which is the Lax operator for the hierarchy. Up to trivial rescalings, our equation is the $E=0$ case of the system
\begin{align*}
q_t &= 4 \, \mathrm{Re} \, \left( 4 \diff^3 q + qw + \diff(qw) - E\diff q \right)  \\
\dbar w &= \diff q
\end{align*}
at a fixed energy $E$. 

\bigskip

With similar notations, the Cauchy problem for mNV is 
\begin{equation}
	\label{mNV}
\left\{
\begin{aligned}    
	&	u_t + (\partial^3  + \dbar^3) u =  N_{mNV}(u), 
	\\
    & u(0) = u_0 \in L^2(\R^2)
\end{aligned}
\right.
\end{equation}
and the cubic nonlinearity is given by 
\begin{align}
    \label{NmNV}
	\frac43 N_{mNV}(u) 
		&=  u \dbar^{-1} (\diff(\ubar \diff u)) 
                +  \diff u \cdot \dbar^{-1} (\diff (|u|^2)) \\
        \nonumber
            &\quad + u \diff^{-1}(\dbar(\ubar \dbar u)) 
                +  (\dbar u) \cdot \diff^{-1} (\dbar (|u|^2))
\end{align}

The mNV equation was introduced by Bogdanov \cite{Bogdanov1987} as a counterpart to the Novikov-Veselov equation in analogy to the mKdV equation as a counterpart to the KdV equation.  Our equation agrees with Bogdanov's up to a $-$ sign and a scale factor in the nonlinearity. The formula given here corrects an error in Perry's paper \cite{Perry2014}, where an incorrect formula was given in equation (1.10) for the nonlinearity in the mNV equation. Details of the correct computation are given in \cite{Perry2025}.
The sign difference in the nonlinearity, compared with Bogdanov,
corresponds to a $\pi/2$ rotation of the coordinates in $\R^2$.

 Like their one-dimensional counterparts, 
 the NV and mNV equations
 are related by a Miura transform
\begin{equation}
    \label{Miura.NV}
    \calM(u) = 2 \diff u + |u|^2.
\end{equation}
The Miura transform \eqref{Miura.NV} maps solutions of mNV 
which satisfy the algebraic condition
\begin{equation}\label{u-constraint}
\diff u = \overline{\diff u}
\end{equation}
to solutions of NV, much as the one-dimensional Miura transform $\calM_1(u) = 2u_x + u^2$ maps solutions of mKdV to solutions of KdV. As we will see, the Miura transform \eqref{Miura.NV} maps Cauchy data for \eqref{mNV} to Cauchy data $q=\calM(u)$ for the NV equation for which $-\Delta+q \geq 0$, suitably interpreted, see Section~\ref{sec:nv}.

Another key part of our story is the fact that the mNV equation \eqref{mNV} is a completely integrable dispersive equation in the hierarchy of the Davey-Stewartson II (DS II) equation, 
which has the form 
\begin{equation}
	\label{DS}
	\left\{
	\begin{aligned}
&		i\diff_t u + (\diff^2 + \dbar^2)u + u(r+ \rbar) = 0\\
&		\dbar r + \diff(|u|^2) =0\\
&		u(0,z) = u_0(z)
	\end{aligned}
	\right.
\end{equation}
meaning that the same scattering transform $\calS$ that linearizes \eqref{DS} also linearizes \eqref{mNV}. We will  recall the definition of $\calS$ in Section~\ref{sec:gns}. This scattering transform can be thought of as a nonlinear Fourier transform, and turns out to satisfy precisely analogous bounds (see Perry \cite{Perry2016}, Brown, Ott, and Perry \cite{BOP2016}, Nachman, Regev, and Tataru \cite{NRT2020}, and earlier references given there).  In particular, the Plancherel theorem and a novel pointwise bound for $\calS$ was proved in \cite{NRT2020}, and used to obtain global existence for solutions to \eqref{DS}. Here, as we will see, for global solvability of mNV we will additionally need an analogue of the Gagliardo-Nirenberg inequality for $\calS$ (see Theorems \ref{t:nonlinGNS} and \ref{t:xtGNS}).

\subsection{ Scaling and criticality}

The mNV equation is invariant with respect to the scaling law 
\[
u(x,y,t) \to \lambda u(\lambda x,\lambda y,\lambda^3 t).
\]
By contrast, the scaling law for DS II is 
\[
u(x,y,t) \to \lambda u(\lambda x,\lambda y,\lambda^2 t).
\]
While different overall, these scaling laws coincide at fixed time, which is consistent
with the fact that they belong to the same 
integrable hierarchy. The Sobolev space 
for the initial data whose norm is invariant 
with respect to this scaling is $L^2(\R^2)$,
which is why we call both of these problems \emph{$L^2$ critical}. This indicates that global
$L^2$ well-posedness  for both of these problems
would be optimal, if possible.

One immediate difficulty which arises in the study of $L^2$ solutions for both mNV and DS II 
is that the nonlinearity is not even well-defined at fixed time for $q \in L^2$. Hence in order to properly interpret the equations
one needs to work in a suitable space-time setting, taking advantage of dispersive type 
norms, such as Strichartz spaces.  This is just borderline for DS II, but much worse in the mNV case. In connection to this issue, we remark that the fact that the two time scalings above are different suggests that the classes of space-time dispersive bounds 
(i.e. Strichartz estimates) for the solutions to mNV, respectively DS II should be different.

A common feature of both problems is the range of Strichartz exponents, which corresponds more generally to two-dimensional 
dispersive flows,
\begin{equation}
	\label{Strichartz.pr}
	\frac{1}{p}+\frac{1}{r} = \frac12, \qquad 2 < p \leq \infty.
\end{equation}
We call such exponents admissible.

However, the number of derivatives differs, due to the different orders in the spatial operators.
For the linear homogeneous DS evolution we have the linear Strichartz estimate 
\begin{equation}
	\label{Strichartz-DS}
			\norm[L^p_t L^r_x]{ u} 
					\lesssim \|u_0\|_{L^2}
\end{equation}
for all admissible pairs $(p,r)$. On the other hand, for the linear homogeneous NV (or mNV) evolution,  we have 
(\cite[Proposition 6.4]{Schottdorf13})
\begin{equation}
	\label{Strichartz-NV}
			\norm[L^p_t L^r_x]{|D|^\frac1p u} 
					\lesssim \|u_0\|_{L^2}
\end{equation}
For the above space we introduce the notation  $S^p$,
\begin{equation}
	\label{Sp}
		\norm[S^p]{u} = \norm[L^p_tL^r_z]{|D|^\frac1p u}, \qquad \frac1p + \frac1r = \frac12, \quad p>2.
\end{equation}

Moving on to the NV system, its scaling law is 
\[
q(x,y,t) \to \lambda^2 q(\lambda x,\lambda y,\lambda^3 t).
\]
Formally this corresponds to the initial data
space $\dot H^{-1}(\R^2)$. However, this space
appears to be too narrow for a well-posedness theory, and in particular would require a formal cancellation condition for the integral of $q$. u
Instead, the Miura map \eqref{Miura.NV} indicates
that a better space should be $\dot H^{-1}(\R^2)+ L^1(\R^2)$. This discussion is expanded in the last section of the paper. The corresponding Strichartz norms are one derivative lower than $S^p$, and are denoted by $DS^p$.

\subsection{ The main theorem }

We begin by reviewing the state of the art 
prior to the present work, beginning with the 
DS-II model. The definitive global well-posedness   result for the DS-II problem, proved in \cite{NRT2020}, is as follows:

\begin{theorem}[Nachman-Regev-Tataru~\protect{\cite[Theorems 1.4 and 1.6]{NRT2020}}]
 \label{thm:NRT}
\ The DS-II \\ problem \eqref{DS} is globally well-posed for 
large initial data $u_0 \in L^2$, with the solution depending (locally)  smoothly on the initial data, and satisfying a global Strichartz bound 
\begin{equation}\label{DS-L4}
\|u\|_{L^4_{xt}} \lesssim    C(\norm[L^2]{u_0}). 
\end{equation}
\end{theorem}

This followed a much earlier 
small  $L^2$ data result of Ghidaglia-Saut~\cite[Theorem 2.1]{GS1990}, which was obtained by treating the nonlinearity perturbatively in Strichartz spaces. 

What is of interest to us here is not so much the result itself but the method of proof, which uses 
the scattering transform $\calS$ associated to DS-II in order to prove the critical $L^4$ 
Strichartz bound in the above theorem.
The scattering transform has the key property that 
it diagonalizes the DS-II flow. Precisely, 
$u$ solves the nonlinear DS-II problem if and only if 
 $\calS u$ solves the  linear equation \eqref{S_DSflow}.
A detailed definition of $\calS$ is provided in Section~\ref{sec:gns}.

Given the above property, and the fact that $\calS$ is an involution,  the large data $L^2$ well-posedness of  DS-II can be obtained as a consequence 
of properties of $\calS$, some of which we recall:
\begin{theorem}[Nachman-Regev-Tataru~\protect{\cite[Theorem 1.2]{NRT2020}}]
 \label{thm:NRT-S}
The scattering transform $\calS$ is a global smooth diffeomorphism on $L^2$, with $\mathcal S^{-1}=\mathcal S$, which also satisfies the scale invariant $L^4$ 
bound
\begin{equation}\label{S-L4}
\| \calS u\|_{L^4} \lesssim  C(\|u\|_{L^2}) \| \hat u\|_{L^4}  
\end{equation}we
whenever $\hat u \in L^2 \cap L^4$.
\end{theorem}
In effect \cite{NRT2020} establishes a more accurate pointwise bound, which has the form
\begin{equation}\label{point-calS}
 |\calS u(k)| \leq C(\norm[L^2]{u_0})M\hat u(k),  
\end{equation}
where $M$ denotes the Hardy-Littlewood maximal function. For the Fourier transform $\hat u(k)$ we use the normalization \eqref{FT} (see
Section~\ref{sec:gns}).

\bigskip

Turning our attention to our main target, namely the mNV system, our starting point is the small $L^2$ data result
stated\footnote{His result applies in effect not directly to mNV but to a related model problem, see Remark~\ref{r:tobias}.} in Schottdorf \protect{\cite[Theorem 6.1]{Schottdorf13}}
:

\begin{theorem}
	 \label{thm:Schottdorf}
	The mNV system \eqref{mNV} is globally well-posed and scatters for small initial data in $L^2(\R^2)$.
\end{theorem}

The proof of this result is still perturbative, 
but it can no longer be done directly using Strichartz spaces, due to an unfavourable balance of derivatives. Instead, Schottdorf's approach relies on the more refined $U^p$ and $V^p$ type spaces associated to the linear mNV flow, see Section~\ref{sec:func} for their definition. This is discussed more extensively in Section~\ref{sec:lwp}.

Our goal here is to study the large data problem, for which a perturbative approach no longer works. Our main result is as follows:

\begin{theorem}
	\label{thm:mNV} 
The modified Novikov-Veselov system is globally well-posed for all initial data $u_0$ in $L^2$, in the following sense:

\begin{description}
    \item[(i) Existence] Given $u_0$ in $L^2$, there exists a global solution 
    \[
    u \in U^2_{NV} L^2(\R^2) \subset C(\R; L^2(\R^2)).
    \]

\item [(ii) Uniqueness] The solution $u$ is unique 
in $U^2_{NV} L^2(\R^2)$.

\item [(iii) Lipschitz dependence] The flow map 
\[
L^2\ni u_0 \longrightarrow u \in U^2_{NV} L^2(\R^2)
\]
is smooth, uniformly on bounded sets in $L^2$.
In particular we have a global bound
\begin{equation}
\| u \|_{ U^2_{NV} L^2} \leq C(\|u_0\|_{L^2})\|u_0\|_{L^2}.   
\end{equation}

\item [(iv) Higher regularity] If in addition $u_0 \in H^s(\R^2)$ for some $s > 0$, then we have $u \in 
U^2_{NV} H^s(\R^2)$, with a global bound
\begin{equation}\label{mNV-U2}
\| u\|_{U^2_{NV} H^s(\R^2)} \leq C(\|u_0\|_{L^2}) \|u_0\|_{H^s(\R^2)}.  
\end{equation}

\item [(v) Scattering] For each $u_0 \in L^2(\R^2)$ there exist $u_\pm \in L^2(\R^2)$, which are expressed in terms of the scattering transform as $u_{\pm} = \widecheck{\calS(u_0)}$, so that
	\begin{equation}
	  \lim_{t \to \pm \infty} \norm[L^2]{u(t) - S_{NV}(t)u_\pm} = 0,   
	\end{equation}
where $S_{NV}(t)$ denotes the corresponding linear flow (see \eqref{mNV.lin}).
 \\
 Furthermore, the wave operators $u_0 \to u_{\pm}$ are global smooth diffeomorphisms of $L^2$.
 \item [(vi) Pointwise bound]
 \begin {equation} \label{ptws-bound}
|u(t,x)| \leq C(\|u_0\|_{L^2})Mu^{\text{lin}}(t,x),
 \end{equation}
where $u^{\text{lin}}(t,.)=S_{NV}(t)u_\pm$, and $M$ is the Hardy-Littlewood maximal function.

\end{description}
\end{theorem}

We remark that, as a consequence of the $U^2$ bound \eqref{mNV-U2}, the solutions also satisfy the expected Strichartz bounds
\begin{equation}
\|u\|_{S^p}\leq C(\|u_0\|_{L^2(\R^2)})\|u\|_{L^2}.
\end{equation}
and in effect a stronger $\ell^2$ Besov version of this.

We also remark that one can easily show that 
$\hat{\calS}$ is a smooth diffeomorphism in all $H^s$ for $s \geq 0$,
as a direct consequence of the DS II results in \cite{NRT2020}.
\subsection{ The Novikov-Veselov equation}

Given  the mNV result in Theorem \ref{thm:mNV},  
we can use the Miura map \eqref{Miura.NV} in order to study the NV equation for initial data in the range\footnote{with a slight abuse of notation as we only consider $L^2$
functions satisfying the constraint \eqref{u-constraint}.}
$\calM(L^2) \subset \dot H^{-1} + L^1$
of the Miura map,  a far larger class of initial data than that considered in \cite{Perry2014}.  
Here we think of $\calM(L^2)$ as a subset of $\dot H^{-1}(\R^2) + L^1$, algebraically and topologically.
We note that the embedding $L^1 \subset \dot H^{-1}$
barely fails, and the $L^1$ component is in effect essential in $\calM(L^2)$. Indeed, if $q = \calM(u)$
with $u \in L^2$ then we have the algebraic identity 
\begin{equation}
\| u\|_{L^2}^2 =  \int_{\R^2} q \, dx    
\end{equation}
where the integral makes sense for all $q \in \dot H^{-1} + L^1$. This immediately shows that within the Miura range the size of $q$ is controlled by its integral,
\begin{equation}
\|q\|_{  \dot H^{-1} + L^1} \lesssim \int q \, dx \ C(\int q \, dx)  \end{equation}
As the Miura map is not surjective, we also seek to describe its range. Heuristically, this range corresponds to the soliton free NV solutions. In Section~\ref{sec:nv} we will prove (see Theorem~\ref{t:AAP}) that a potential $q \in \dot H^{-1} + L^1$ is in the range of $\calM$ if and only if the Schr\"odinger operator $H_q=-\Delta + q$ is nonnegative. This theorem is of independent interest as it yields a sharp novel result of Agmon–Allegretto–Piepenbrink type in a scale invariant homogeneous Sobolev space.
In particular, it will also follow that the condition $H_q=-\Delta + q \geq 0$ is preserved by the NV flow. 
\begin{theorem}
    \label{thm:NV}
The Novikov-Veselov equation is globally well-posed for 
initial data in $\calM(L^2) \subset H^{-1}(\R^2) + L^1(\R^2)$,
in the sense that the data to solution map 
\[
L^2 \cap \calM(L^2) \ni q_0 \to q \in C(\R,L^2\cap \calM(L^2) )
\]
admits a continuous extension 
\begin{equation}
  \calM(L^2) \ni q_0 \to q \in C(\R, \calM(L^2)).  
\end{equation}    

Further, the solutions satisfy global dispersive bounds,
\begin{equation}\label{NV-Str}
\| q\|_{DS^p} \lesssim C(\int q_0\, dx ) \int q_0 \, dx, \qquad 2 <p <\infty,   
\end{equation}
also with continuous dependence on the initial data.
\end{theorem}

We make the following remark on the continuous dependence on the initial data, comparatively for the mNV and the NV flows.
In the mNV case, our result gives Lipschitz, indeed even analytic dependence of the solution on the initial data, which indicates a semilinear nature 
of the flow in $L^2$. However, in the NV case we only obtain continuous dependence, which stems from the 
mere continuity of the inverse of the Miura map
from $\dot H^{-1}+L^1$ into $L^2$. This may possibly indicate a quasilinear nature of the NV evolution in the critical space.

\subsection{Prior Results}
Historically the NV and the mNV equations have been considered from two essentially disjoint perspectives, (i) the completely integrable approach
and (ii) the nonlinear dispersive pde perspective. It is part of the goal of this paper to bring these two aspects together toward the study of the large data problems.

On the integrable side,  modified Novikov-Veselov equation was introduced by Bogdanov \cite{Bogdanov1987}, and belongs to the integrable hierarchy of the DS II equation. In \cite{Perry2014}, Perry showed that the direct and inverse scattering maps for the DS II equation (as formulated in \cite{Perry2016}) could be used to solve the mNV equation for initial data $u_0 \in S(\R^2)$.

The completely integrable structure of the  NV equation was elucidated by Manakov \cite{Manakov76}. Grinevich \cite{Grinevich2000}, Grinevich and Manakov \cite{GM1986}, and Grinevich and S.\ P.\ Novikov, 
\cite{GN1998a,GN1988b} developed inverse scattering for the NV equation at nonzero energy (see also their comments about the zero-energy case in \cite[Supplement 1.I]{GN1988b} and \cite[\S 7.3]{Grinevich2000} ). 

Inverse scattering for the NV equation at zero energy was studied by Boiti, Leon, Manna and Pempinelli \cite{BMP1987}, Tsai \cite{Tsai1993,Tsai1994}, Lassas, Mueller, and Siltanen \cite{LMS2007}, Lassas, Mueller, Siltanen and Stahel \cite{LMSS2012b,LMSS2012a}, Perry \cite{Perry2014}, and Music-Perry \cite{MP2018}. Perry \cite{Perry2014} exploits the connection between the mNV and NV equations via the Miura map \eqref{Miura.NV} (see Bodganov \cite{Bogdanov1987} and also Dubrovsky-Gramolin \cite{DG2008} ) to construct weak solutions of the NV equation for certain initial data of ``conductivity type''--namely those that arise in the transformation of the divergence form operator $\nabla\cdot(\sigma\nabla)$ (with $\sigma$ modeling an electric conductivity) to a Schr\"{o}dinger operator. Such potentials take the form $q=\Delta \psi/\psi$ for a positive function $\psi$, and also arise in other related problems, such as the inverse problem of Calder\'{o}n. For instance, Nachman \cite{Nachman96} showed that such potentials are precisely those with a non-singular scattering transform, and used this to prove unique reconstruction of the conductivity on a bounded domain in $\R^2$ from the corresponding Dirichlet to Neumann map. 

Perry's approach to the NV equation exploits the fact that the mNV equation belongs to the completely integrable hierarchy of the Davey-Stewartson II equation and draws on Perry's analysis of inverse scattering for the DS II equation for initial data in $H^{1,1}(\R^2)$ \cite{Perry2016}. In a similar spirit, Theorem \ref{thm:mNV} of the current paper exploits the more recent results of \cite{NRT2020}, which analyzes the DS II scattering transform as a nonlinear diffeomorphism on $L^2(\R^2)$ and proves global well-posedness of the DS II equation on $L^2(\R^2)$.

There are also a number of results on existence and uniqueness for the mNV and the NV equation from the PDE point of view. 
Angelopoulos \cite{Angelopoulos16} proved local wellposedness of NV in $H^s(\R^2)$ for $s>\frac12$, and local wellposedness for mNV in $H^s(\R^2)$ for $s>1$. 
More recently, Adams and Gr\"{u}nrock \cite{AG2023} proved local well-posedness for the NV equation on $\T^2$ for initial data in $H^s(\T^2)$, $s > -\frac15$, and for the NV equation on $\R^2$ for initial data in $H^s(\R^2)$, $s> -\frac34$. On the global side, Schottdorf~\cite{Schottdorf13} proved 
small data global well-posedness for mNV in the critical space $L^2$;
as noted earlier, this represents a starting point for the present work, where we prove large data global well-posedness. 
On the other hand for NV there are explicit rational solutions blowing up in finite time \cite{TT}, so a general large data global result cannot hold; instead, here we prove global well-posedness for solutions of NV, but in a specified subset of $\dot{H}^{-1}(\R^2) + L^1(\R^2)$. 

\subsection{ An outline of the paper}

The discussion of our results and of our strategy for the proof naturally begins 
with the scattering transform $\calS$;
for a full definition see Section~\ref{sec:gns}.
This linearizes both the DS II and the mNV flows, as follows: 
\begin{itemize}
    \item If $u$ solves the DS II equation \eqref{DS} then $\calS u$ solves
  \begin{equation}\label{S_DSflow}
  i \partial_t (\calS u)(k) = -\omega_{DS}(k) (\calS u)(k),
  \qquad 
  \omega_{DS}(k) := k^2+\bar k^2       
  \end{equation}  
  
   \item If $u$ solves the mNV equation \eqref{mNV} then $\calS u$ solves
  \begin{equation}\label{S_mNVflow}
  i \partial_t (\calS u)(k) = 
  \omega_{NV}(k) (\calS u)(k),
  \qquad 
  \omega_{NV}(k) := k^3+ \bar k^3       
  \end{equation}    
\end{itemize}
With a scattering transform and its inverse at our disposal, one may formally solve both equations globally by setting 
\begin{equation}\label{DS-calS}
u_{DS}(t) = \calS^{-1} e^{it \omega_{DS}(D)} \calS u_0
\end{equation}
for DS-II, respectively
\begin{equation}\label{mNV-calS}
u_{mNV}(t) = \calS^{-1} e^{-it \omega_{NV}(D)} \calS u_0
\end{equation}
for mNV. 

A fundamental result proved in \cite{NRT2020}
asserts that $\calS$ is a global smooth diffeomorphism in $L^2$. This makes the two 
formulas above \eqref{DS-calS} and \eqref{mNV-calS} meaningful for $L^2$ data. However, it is far from the end of the story, as it does not even show in what sense the functions $u_{DS}$ and $u_{mNV}$ defined above solve the DS II, respectively the mNV equations.

The next clue comes from the small data results, which in both cases are proved perturbatively using the Duhamel formulation of the equations, in well chosen function spaces associated to the corresponding linear 
dispersive flows. In the case of the DS II
flow, the key spacetime norm for solutions is the Strichartz $L^4_{xt}$ norm. Controlling this norm was  achieved in  Theorem~\ref{thm:NRT-S}, which in particular implies
 the 
fixed time bound 
\begin{equation}\label{L4}
\| u \|_{L^4_x} \leq C(\|u\|_{L^2_x}) \| \widehat{\calS u}\|_{L^4_x}    
\end{equation}
which by linear Strichartz estimates implies 
the space-time bound
\[
\begin{aligned}
\| u\|_{L^4_{xt}} \leq & \  C(\|u\|_{L^\infty L^2}) \| \widehat{\calS u}\|_{L^4_{xt}} 
\leq C(\|u_0\|_{L^2}) \| e^{it \omega_{DS}(D)}\widehat{\calS u_0}\|_{L^4_{xt}} 
\\
\leq & \ C(\|u_0\|_{L^2}) \|  u(0)\|_{L^2_{x}}.
\end{aligned}
\]
This in turn suffices in order to justify  the Duhamel formulation of the DS II equation 
in $L^4_{xt}$, with global bounds depending only on the $L^2$ size of the data.

Now we turn our attention to our main subject in this paper, namely the mNV equation. 
Since the dispersion relation is different, 
so are the corresponding Strichartz norms. 
In particular, the $L^4_{xt}$ norm should be replaced by $S^4:=L^4_t \dot W^{\frac14,4}$,
and \eqref{L4} should be replaced by 
\begin{equation}\label{W4}
\| u \|_{\dot W^{\frac14,4}_x} \leq C(\|u\|_{L^2_x}) \| \widehat{\calS u}\|_{\dot W^{\frac14,4}_x}    
\end{equation}
At this point we have arrived at the two main difficulties we are facing in this paper, and which did not exist in the DS II case studied in \cite{NRT2020}:

\begin{enumerate}[label=(\roman*)]
\item We do not know whether \eqref{W4} holds,
as well as its counterparts using other mixed norms of Strichartz type.

\item Even if \eqref{W4} were true, this 
would not suffice in order to close the 
fixed point argument in Schottdorff's thesis
\cite{Schottdorf13}; instead, his argument is carried out in  $U^2_{NV}$ type spaces, 
which seem completely hopeless from the perspective of estimates of the form \eqref{W4}. Furthermore, the $U^2_{NV}$ spaces lack a key property, namely \emph{time divisibility}, which is essential in the small data to large data transition.
\end{enumerate}
Now that it is clear what issues need to be addressed, we are ready to proceed with the 
outline of the paper.

In Section~\ref{sec:func} we begin by setting up our notations. More importantly, we review the linear theory for both the DS II and the NV case, and in particular we define the Strichartz spaces $S^p$ as well as the $U^2_{NV}$ and the $V^2_{NV}$ spaces.

In Section~\ref{sec:lwp} we revisit and provide
a key refinement to Schottdorff's small data argument, which roughly resolves the difficulty (ii) above. We achieve there 
two objectives: 
\begin{itemize}
    \item we first prove a large data local well-posedness result
    \item we show that the $U^2_{NV}$ size of the local solutions is controlled by the 
    $L^2$ data size with implicit constants depending only on its Strichartz $S^p$  
    norm, for well chosen $p$. We will refer 
    to this norm as the \emph{control norm}.
\end{itemize}
In particular, the lifespan of the local solutions is shown to depend only on the 
$S^p$ norm distribution for the solution to the corresponding linear flow. This in turn guarantees that solutions can be continued for as long as their $S^p$ norm remains finite.

In Section~\ref{sec:gns} we turn our attention to the scattering transform, whose 
definition we recall. Our primary objective 
there is to find a suitable replacement 
for the (unknown) bound \eqref{W4}. Our solution is to prove a slightly weaker 
interpolation bound, which we call a nonlinear Gagliardo-Nirenberg inequality. In the context of \eqref{W4},
this has the form
\begin{equation}\label{W4-GNS}
\| u \|_{\dot W^{\frac18,\frac83}_x} \leq C(\|u\|_{L^2_x})\|u\|_{L^2}^\frac12 \| \widehat{\calS u}\|_{\dot W^{\frac14,4}_x}^\frac12,
\end{equation}
though our complete result covers a fuller range of indices. As it turns out, this suffices 
in order to yield global in time bounds 
for our control norm discussed earlier, and thus resolves the difficulty (i) noted above.

Armed with the two ingredients above, in Section~\ref{sec:gwp} we use the time divisibility of $S^p$ norms in order to complete the proof of our main global well-posedness result for the mNV flow in Theorem~\ref{thm:mNV}.

In the final section of the paper, we explore 
the consequences of our mNV result for the NV
flow, by exploiting the Miura map and it inverse. To characterize the range of the Miura map, we provide a proof of a sharp, scale invariant form of the  Agmon–Allegretto–Piepenbrink principle.

\subsection{ Acknowledgements}
The first author was supported by the NSERC Discovery Grant RGPIN-06235
.The second author was supported in part by
a grant from the Simons Foundation (359431, PAP).
The third author was supported by the NSF grant DMS-2054975 as well as by a Simons Investigator grant from the Simons Foundation
 and a Simons Fellowship. The authors thank Idan Regev for his work in the early stages of this project and Barry Simon for helpful correspondence on Agmon-Allegretto-Piepenbrink theory.

\section{Preliminaries}
\label{sec:func}

We denote by $\calF$ the Fourier transform
$$ (\calF f)(\xi) = \int_{\R^2} e^{-ix\cdot \xi} f(x) \, dx. $$
If $\varphi$ is a radial bump function with $\varphi(\xi)=1$ on $|\xi| \leq 1$ and $\varphi(\xi)=0$ for $|\xi| \geq 2$, we set $\psi(\xi)=\varphi(\xi) - \varphi(2\xi)$ and, for dyadic integers $K=2^k$ we set $\psi_k(\xi) = \psi(\xi/2^k)$. The Littlewood-Paley multipliers are the Fourier multipliers $P_k f= \calF^{-1}( \psi_k \calF f)$. We will also use the localizations
	$$P_{\ell \leq \dotarg \leq m} = \sum_{k=\ell}^m P_k $$
	and similar localizations onto a range of $k$'s.

\subsection{Function Spaces}

We denote by $\dot{B}^{s,p}_{q}$ the homogeneous Besov space with norm
\begin{equation}
	\label{Besov}	
		\norm[\dot{B}^{s,p}_{q}]{f} = 
	\left(	\sum_{k \in \Z} (2^{ks}\norm[L^p]{P_k f})^q	\right)^{1/q}
\end{equation}
 where, for each $k \in \Z$,  $P_k$ is a Littlewood-Paley operator smoothly localizing in Fourier space to $2^{k-1} \leq |\xi| \leq 2^{k+1}$.

Motivated by Strichartz estimates for $S_{NV}(t)$, we define a space $S^p$ with norm
$$ \norm[S^p]{u} = \norm[L^p_t L^r_x]{|D|^{1/p}u}, \qquad \frac{1}{p}+\frac{1}{r} = \frac12$$ 
where $D$ is the derivative in the spatial variables only. 

Next we recall the spaces $U^p$ and $V^p$ and the adapted function spaces $U^p_S$ and $V^p_S$ associated to a unitary evolution. The space $V^p$ was introduced by Wiener in \cite{Wiener24}. The dispersive versions of spaces, adapted to the unitary $S$ flow, were first introduced in unpublished work of the last author~\cite{T-unpublished}, in order to have better access to both Strichartz estimates and bilinear $L^2_{t,x}$ bounds,  in particular without losing endpoint estimates; see also Koch-Tataru~\cite{KT2005}, Herr-Tataru-Tzvetkov~\cite{HTT}, Hadac-Herr-Koch \cite{HHK2009,HHK2010} and Candy-Herr \cite{CH2018} for some of the first applications of these spaces. 
	
	If $\calB$ is a Banach space and $I \subseteq \R$ is an open interval, we say that $u:I \to \calB$ is \emph{ruled} if $u$ has left- and right-hand limits at each point $t$ of $I$. If $I=(a,b)$ for $-\infty \leq a < b \leq \infty$, a partition of $I$ is a finite monotone sequence 
	$$ a< t_1 < \ldots < t_N = b. $$
	We denote the set of all partitions by $\scrT$. 
	
	For $1 < p < \infty$ the space $V^p(I)$ consists of ruled functions $v:I \to \calB$ for which the norm
	$$ \norm[V^p]{u} = \sup_{\tau \in \scrT}
			\left( 
				\sum_{j=1}^{N-1} \norm[\calB]{v(t_{j+1})-v({t_j})}^p
			\right)^\frac1p
	$$
	is finite, where by convention $v(b)=0$. 
	
	The space $U^p(I)$ is an atomic space whose $p$-atoms $a$ take the form 
	$$a(t) = \sum_{j=1}^{n-1} \chi_{[t_{j},t_{j+1})}(t) \psi_i $$
	where 
	$$\sum_{i=1}^n \norm[\calB]{\psi_i}^p \leq 1.$$
	$U^p(I)$ consists of functions of the form
	$$ u(t) = \sum_{n=1}^\infty \lambda_n a_n(t), \quad \sum_{n=1}^\infty |\lambda_n| < \infty $$
	where $\lambda_n \in \C$ and each $a_n$ is a $p$-atom. We equip $U^p(I)$ with the norm
	\begin{equation}
		\label{Up.norm}
			\norm[U^p]{u} = 
				\left\{
					\inf 
					\sum_{n=1}^\infty |\lambda_n|:
					u(t) = \sum_{n=1}^\infty \lambda_n a_n(t)
					\text{ where each }a_n \text{ is a $p$-atom}
				\right\}.
	\end{equation}
	We have the continuous embedding $U^p(I) \subset V^p(I)$. 
	
	For $1< p < \infty$, we define the space 
	\begin{equation}
		\label{DUp}
			DU^p = \{u': u \in U^p \},
	\end{equation}
	where $u'$ denotes the distribution derivative of $u(t)$ with respect to $t$, with the induced norm. We have \cite[Theorem B.18]{KT2018}:
	
	\begin{theorem}
		\label{thm:DUp}
		For all distributions $f$,
		\begin{equation}
			\label{DUp.norm}
			\norm[DU^p]{f} = \sup 
				\left\{ 
					\int f \phi \, dt: 
					\norm[V^q]{\phi} \leq 1, 
					\phi \in C_0^\infty(\R,\calB) 
				\right\}.
		\end{equation}
	\end{theorem}	
	If $I=\R$, $\calB$ is a Hilbert space and $S(t)$ is a one-parameter unitary group, we define spaces 
		$$U^p_S = \{ u: S(-t)u(t) \in U^p \}$$ and 
		$$V^p_S = \{ v: S(-t)v(t) \in V^p \}$$ 
	with norms
	\begin{align}
		\label{UpS.norm}
		\norm[U^p_S]{u} &= \norm[U^p]{S(-\dotarg) u(\dotarg)}\\	
		\label{VpS.norm}
		\norm[V^p_S]{v}	&=	\norm[V^p]{S(-\dotarg) v(u\dotarg)}.
	\end{align}
	Note that $U^2_S \subset V^2_S$ and we have the continuous embedding
	\begin{equation}
		\label{Ups.VpS}	
			\norm[V^2_S]{u} \lesssim \norm[U^2_S]{u} 
	\end{equation}

\subsection{Strichartz Estimates 
for linear flows}

In this subsection we review the Strichartz and $U^p$/$V^p$ spaces and estimates. For completeness we consider both the DSII case and the $NV$ case, whose associated 
linear inhomogeneous flows have the form
\begin{equation}
	\label{DS-lin-re}
	\left\{
	\begin{aligned}
&		i\diff_t q + (\diff^2 + \dbar^2)q  = f\\
&		q(0) = q_0
	\end{aligned}
	\right.
\end{equation}
respectively 
\begin{equation}
    \label{NV-lin}
\left\{
\begin{aligned}
   & u_t + (\partial^3 + \dbar^3) u =  g,\\
    \nonumber
    & u(0) = u_0.
 \end{aligned}
\right.
\end{equation}   

The range of Strichartz exponents is the same for both
evolutions, corresponding to two-dimensional 
dispersive flows,
\begin{equation}
	\label{Strichartz.pr-re}
	\frac{1}{p}+\frac{1}{r} = \frac12, \qquad 2 < p \leq \infty.
\end{equation}
We call such exponents admissible.
However, the number of derivatives varies, due to the different orders in the spatial operators.

For the linear evolution \eqref{DS-lin-re} we have the linear Strichartz estimate 
\begin{equation}
	\label{Strichartz-DS-re}
			\norm[L^p_t L^r_x]{ q} 
					\lesssim \|q_0\|_{L^2}+  \norm[L^{p_1'}L^{q_1'}]{f},
\end{equation}
for all admissible pairs $(p,q)$ and $(p_1,q_1)$.

On the other hand for the linear NV evolution \eqref{NV-lin} we have a different balance of derivatives
\begin{equation}
	\label{Strichartz-NV-re}
			\norm[L^p_t L^r_x]{|D|^\frac1p u} 
					\lesssim \|u_0\|_{L^2}+  \norm[L^{p_1'}L^{q_1'}]{|D|^{-\frac1p} g},
\end{equation}
where, matching the last estimate, we introduce the notation  $S^p$ for the space of functions with the norm on the left.

Just using Strichartz norms does not suffice for our purposes here, instead we also need the associated 
$U^p$ and $V^p$ spaces. We denote the $L^2$ groups 
of isometries for the two flows by $\{S_{DS}(t)\}$, respectively $\{S_{NV}(t)\}$. Correspondingly, we have 
the associated $U^p$ and $V^p$ spaces, for which we use the notations $U^p_{DS}$, $V^p_{DS}$, respectively
$U^p_{NV}$, $V^p_{NV}$. These spaces will be used for the solutions to the two flows. On the other hand for 
the source terms we use the spaces 
\[
DU^2_{NV} = \{  u_t + (\partial^3 + \dbar^3) u; \ u \in U^2_{NV}\}, 
\]
and similarly for DSII. By Theorem~\ref{thm:DUp}, 
these spaces are related to $V^2$ by duality,
\begin{equation}
(DU^2_{NV})' = V^2_{NV}.
\end{equation}
We also make the important observation that 
\[
\| \bu\|_{U^2_{NV}} = \|u\|_{U^2_{NV}},
\]
which is due to the fact that the symbol of the linear NV operator is odd.

In this context, the linear bound for the linear NV flow \eqref{NV-lin} reads as follows:
\begin{equation}\label{L2-solve}
\|u\|_{U^2_{NV}} \lesssim \|u_0\|_{L^2} + \| g\|_{DU^2_{NV}}
\end{equation}    
This will be our main linear bound for the NV
flow.
One can see that this bound is strictly stronger than
the Strichartz bound \eqref{Strichartz-NV-re}, due to the 
embeddings
\begin{equation}\label{UV-to-Sp}
 U^2_{NV} \subset V^2_{NV} \subset S^p, \qquad (S^{p})' \subset DU^2_{NV}.    
\end{equation}

On the above space-time function spaces it will be convenient to superimpose an $\ell^2$ Besov structure.
For clarity we use a slightly nonstandard notation, 
$\ell^2 S^p$, $\ell^2 U^2_{NV}$, etc. For instance
\begin{equation}
    \|u\|_{\ell^2 S^p}^2 : = \sum_k  \|P_k u\|_{ S^p}^2
\end{equation}
and similarly for the other spaces. On occasion we will also replace $\ell^2$ by $\ell^q$ for $1 \leq q \leq \infty$. In particular we will use the shortcuts $\dot X$, $\dot Y$ for the 
closures of $\calS(\R \times \R^2)$ in the norms
$\ell^2 U^2_{NV}$ respectively  $\ell^2 V^2_{NV}$. 
Then we have
\begin{equation}
	\label{dotXY}
		 \norm[\dot{Y}]{u} \lesssim \norm[\dot{X}]{u}
\end{equation}

The corresponding space of source terms is denoted by $\DX$, and the counterpart of the estimate \eqref{L2-solve} is 
\begin{equation}\label{L2-solve-X}
\|u\|_{\dot X} \lesssim \|u_0\|_{L^2} + \| g\|_{\DX}.
\end{equation}
For higher Sobolev regularity we will use the norms $\dot X^s$, $\dot Y^s$ defined by 
\begin{equation}
\| u\|_{\dot X^s}^2 = \sum_k 2^{2ks} \|P_k u\|_{U^2_{NV}}^2, \qquad \| v\|_{\dot Y^s}^2 = \sum_k 2^{2ks} \|P_k v\|_{V^2_{NV}}^2.
\end{equation}

\subsection{Bilinear \texorpdfstring{$L^2$}{L2} estimates}

Here we consider two frequency localized functions at dyadic frequencies $L$ and $M$.
We first note the straightforward bound
\begin{equation}
	\label{St.bi1}
		\norm[L^2]{P_L (u) P_M (v)} \lesssim L^{-\frac1p} M^{-\frac1r}
			\norm[S^p]{u} \norm[S^r]{v}, \qquad 
            \frac{1}{p} + \frac{1}r = \frac12.
\end{equation}
For solutions to the homogeneous equation \eqref{NV-lin} we have access to all $S^p$ 
norms, so it is natural to choose the indices $p$ and $r$ in order to improve the constant.
However, we can do even better than that, namely the bilinear dyadic $L^2$ bound
\begin{equation}
	\label{bi-hom}
		\norm[L^2(\R \times \R^2)]{S_{NV}(t) u_{0,\lambda}  S_{NV}(t) u_{0,\mu}} \lesssim \mu^\frac12 \lambda^{-1}
			\norm[L^2]{u_{0,\lambda}} \norm[L^2]{u_{0,\mu}}
		\qquad \mu \lesssim \lambda.
\end{equation}
This then easily extends to $U^2_{NV}$ functions,
\begin{equation}
	\label{bi-UU}
		\norm[L^2(\R \times \R^2)]{u_\lambda u_\mu} \lesssim \mu^\frac12 \lambda^{-1}
			\norm[U^2_{NV}]{u_\lambda} \norm[U^2_{NV}]{u_\mu}
		\qquad \mu \lesssim \lambda,
\end{equation}
and will be very useful in the proof of our results. The same also works for $V^2_{NV}$ with a $\log(\lambda/\mu)$ loss
\begin{equation}
	\label{bi-UV}
		\norm[L^2(\R \times \R^2)]{u_\lambda u_\mu} \lesssim	
			\frac{\mu^\frac12}{\lambda} 
				\left( 1 + \log(\lambda/\mu)\right)
				\norm[V^2_{NV}]{u_\lambda} \norm[U^2_{NV}]{u_\mu}
		\qquad \mu \lesssim \lambda.
\end{equation}
and also with $U^2_{NV}$ and $V^2_{NV}$ norms interchanged.

\subsection{Lateral Strichartz norms}

Another viewpoint on bilinear $L^2$ estimates can be 
gained by using lateral Strichartz norms, where the role of the time variable and one spatial variable are interchanged.
To set up the notations, given a unit vector $e$ we 
consider an associated orthonormal frame $(e,e^{\perp})$
in $\R^2$, and corresponding coordinates $(x_e,x_e^\perp)$
given by 
\[
x_e = x \cdot e, \qquad x_e^\perp = x \cdot e^\perp. 
\]
Then for any pair $(p,q)$ of Strichartz exponents we define the lateral mixed norm spaces 
\[
L^p_e L^q : = L^p_{x_e} L^{q}_{t,x_e^\perp}, 
\]
and we ask whether linear NV solutions can be measured in these spaces. The case $p=q = 4$ is trivial, as then the lateral norms 
agree with the standard Strichartz norms. 

The energy norm $(p,q) = (\infty, 2)$ deserves special attention.
Since we cannot interpret the NV equation as an  evolution equation in the $x_\omega$ direction, it follows that we cannot control this norm globally. However, we 
can control it for restricted conical sets of frequencies where 
the group velocity in the $\omega$ 
direction is nontrivial:

\begin{proposition}
Assume that $u_0$ is frequency localized in the region
\begin{equation}\label{Ce}
C_e = \{ \xi \in \R^2; \ |\nabla \omega_{NV} \cdot e| \geq c |\xi|^2\}   
\end{equation}
Then for all Strichartz pairs $(p,q)$ with $4 \leq p \leq \infty$
we have
\begin{equation}\label{to-energy}
\||D|^{1-\frac{3}{p}} S_{NV}(t) u_0\|_{L^p_e L^q} \lesssim \|u_0\|_{L^2}
\end{equation}
\end{proposition}
Here we included the full family of estimates which follows by interpolating the lateral energy estimate with the already known $L^4$ bound. The lateral energy estimate, on the other hand, is immediate.

On the other side of $p= 4$, the direction of the group velocity 
no longer plays a role, and the estimate only hinges on having full dispersion; see the corresponding discussion in 
\cite{T-nlw} for the wave equation, respectively \cite{BIKT} for the Schr\"odinger equation. So we have
\begin{proposition}
For all Strichartz pairs $(p,q)$ with $2 < p \leq 4$
we have
\begin{equation}\label{to-Pecher}
\||D|^{1-\frac{3}{p}} S_{NV}(t) u_0\|_{L^p_e L^q} \lesssim \|u_0\|_{L^2}
\end{equation}
\end{proposition}
We denote by $S^p_e$ the spaces with the norms on the left.

We can decompose the Fourier space into finitely many conical regions 
where \eqref{to-energy} holds; it suffices in effect to work with just two directions $e_1$ and $e_2$. Combining this with \eqref{to-Pecher} this yields a better balance of frequencies in the bilinear $L^2$ bound,
\begin{equation}
	\label{bi-hom-}
		\norm[L^2(\R \times \R^2)]{S_{NV}(t) u_{0,\lambda}  S_{NV}(t) u_{0,\mu}} \lesssim \mu^{-1+\frac{3}p} \lambda^{-1+\frac3r}
			\norm[L^2]{u_{0,\lambda}} \norm[L^2]{u_{0,\mu}}.
\end{equation}
In the unbalanced case say $\mu \ll \lambda$ this is much better than the balance obtained in \eqref{St.bi1}, and almost yields 
\eqref{bi-hom}. The latter would correspond to $p = 2$ and $r=\infty$, which is not obtained
only because the Strichartz endpoint $(2,\infty)$ is forbidden.

In the same manner as in \eqref{UV-to-Sp}, we note  the 
embeddings
\begin{equation}\label{UV-to-Sp-P}
 U^2_{NV} \subset V^2_{NV} \subset S^p_e, \qquad (S^{p})' \subset DU^2_{NV}, \qquad 2 < p \leq 4, 
\end{equation}
respectively 
\begin{equation}\label{UV-to-Sp-E}
 P_e U^2_{NV} \subset P_e V^2_{NV} \subset S^p_e, \qquad P_e (S^{p}_e)' \subset DU^2_{NV},   \qquad 4\leq  p \leq \infty 
\end{equation}
where $P_e$ is the frequency projector to the set $C_e$ in \eqref{Ce}.

\section{Local Well-Posedness and Continuation for mNV}
\label{sec:lwp}

\subsection{Schottdorf's small data result}
As noted earlier, our starting point here is  Theorem~\ref{thm:Schottdorf}, which proves global existence of solutions and scattering for the mNV equation \eqref{mNV} with $L^2$ initial data of small norm using fixed point arguments on the Duhamel form of \eqref{mNV}. 

\begin{remark}\label{r:tobias}
We note that, while Schottdorf's  result is stated in 
\cite{Schottdorf13} for mNV, the proof there is only carried out in the case of a simplified nonlinearity 
of the form $N(u) = u^2 \partial u$. Unfortunately, his  auxiliary argument reducing the estimates for the full problem to the simplified case is incorrect.
\end{remark}
For later 
reference we provide a brief outline of his argument for the above simplified model problem.
Denote by $S_{NV}(t)$ the solution operator for the linear problem
\begin{equation}
	\label{mNV.lin}
	\left\{
	\begin{aligned}
	v_t &= -(\diff^3 + \dbar^3) v,\\
	v(0) &= v_0 \in L^2(\R^2)
	\end{aligned}
	\right.
\end{equation}
The Duhamel form of  \eqref{mNV} is then
\begin{equation}
	\label{mNV.Duhamel}
	u(t) = S_{NV}(t) u_0 + \int_0^t S_{NV}(t-s) N_{mNV}(u)(s) \, ds.
\end{equation}
Schottdorf's argument is a fixed-point argument 
in the space $\dot X$.  Schottdorf seeks solutions of \eqref{mNV} in the space $\dot{X}$, with the source term (i.e. the nonlinearity) estimated in the space $\DX$, which can be seen as the dual of $\dot Y$. 
The linear part of his analysis relies on the bound \eqref{L2-solve-X}, while the nonlinear part requires  trilinear estimates on $N_{mNV}(q)$. As $\dot{X} \subset \dot{Y} $ and $\lim_{t \to \pm \infty} S_{NV}(-t) u(t)$ exists in $L^2$ for functions $u \in \dot{Y}$, he also obtains scattering operators $W_\pm$.

His trilinear estimate asserts the Duhamel term is a multilinear map from $\dot{X} \times \dot{X} \times \dot{X}$ to $\DX$.  The frequency-localized nonlinearity has to be placed in $DU^2$ and the main estimate is a cubic estimate
\begin{equation}
	\label{tri}
	\norm[\DX]{u_1 u_2 \diff_x u_3} \lesssim \norm[\dot X]{u_1} \norm[\dot X]{u_2} \norm[\dot X]{u_3}
\end{equation}
or by duality
\begin{equation}
	\label{tri-dual}
	\left| \int u_1 u_2 \diff_x u_3 v \, dxdt\right| \lesssim 
		\norm[\dot X]{u_1} \norm[\dot X]{u_2} \norm[\dot X]{u_3} \norm[\dot Y]{v}.
\end{equation}
Two main tools in the proof of \eqref{tri-dual} are  the bilinear dyadic $L^2$ bounds \eqref{bi-UU} and \eqref{bi-UV}.

\subsection{Large data local well-posedness}
Our first goal here is to consider large-data result for this problem, and establish the following local well-posedness and continuation result.

\begin{theorem}
	\label{thm:mNV.global}
a) 	The mNV system \eqref{mNV} is locally well-posed for large initial data in $L^2$, in the sense that for each initial data $u_0$ there exists some time $T > 0$
and a unique local solution 
$u \in \dot X[0,T]$, depending smoothly\footnote{This means that there exists $T > 0$ and a neighbourhood  $B$ of $u_0$ in $L^2$ so that for each initial data in $B$, the solution exists in $[0,T]$ and depends smoothly on $u_0$.} on $u_0$.

b) The solutions can be continued for as long as $\norm[S^p]{u}$ is finite. In particular, if $\norm[L^2]{u_0}=R$, then for any $2 < p < \infty$ we have the global bound
\begin{equation}
		\label{uSp}
			\norm[\dot{X}[0,T{]}]{u} \lesssim R(1+ R \norm[S^p]{u})^{\frac{p}{2	}}.
	\end{equation}
c) If in addition $u_0 \in H^s$ for some $s > 0$, then the solution $u$ also remains in $H^s$, with the bound
\begin{equation}
		\label{uSp-s}
			\norm[\dot{X^s}[0,T{]}]{u} \lesssim \| u_0\|_{\dot H^s} (1+ R \norm[S^p]{u})^{\frac{p}{2	}}.
	\end{equation}

\end{theorem}

We note two simple consequences of the 
theorem and  of its proof. The first one is an intermediate step in the proof, and 
a natural extension of the ``small data'' setting:

\begin{corollary}
 Assume that $T$ is chosen so that 
 \begin{equation}
 \| S_{NV}(t) u_0\|_{\ell^\infty S^p[0,T]} \ll \|u_0\|_{L^2}^{-1}.    
 \end{equation}
 Then the solution $u$ exists in $[0,T]$
 and satisfies
	 \begin{equation}
		\label{uSp-bdd}
			\norm[\dot{X}]{u} \lesssim \|u_0\|_{L^2}.
	\end{equation}
 and 
  \begin{equation}
		\label{uSp-low}
			\norm[{S^p[0,T]}]{u} \lesssim 
            \| S_{NV}(t) u_0\|_{S^p[0,T]}.
	\end{equation}
\end{corollary}

The second one requires the full strength 
of the theorem, but applied on a time interval where some smallness holds:

\begin{corollary}
 Assume that $u$ is an $L^2$ solution in $[0,T]$ so that 
 \begin{equation}
 \| u\|_{\ell^\infty S^p[0,T]} \ll \|u_0\|_{L^2}^{-1}.    
 \end{equation}
 Then the solution $u$ satisfies \eqref{uSp-bdd} and  
  \begin{equation}
		\label{uSp-low+}
		\| S_{NV}(t) u_0\|_{S^p[0,T]}
        \lesssim 
        \norm[{S^p[0,T]}]{u}. 
\end{equation}
\end{corollary}

Closely related to the above result for the Cauchy problem \eqref{mNV}, we we can formulate a corresponding statement for the scattering problem, with data at $\infty$,
\begin{equation}
	\label{mNV-scattering}
\left\{
\begin{aligned}    
	&	u_t + (\partial^3  + \dbar^3) u =  N_{mNV}(u), 
	\\
    & \lim_{t \to \infty} S_{NV}(-t) u(t)  = u_+ \in L^2(\R^2)
\end{aligned}
\right.
\end{equation}

\begin{theorem}
	\label{thm:scattering}
 	The scattering problem \eqref{mNV-scattering} for the mNV system  is locally well-posed for large initial data in $L^2$, in the sense that for each scattering data $u_+$ there exists some time $T > 0$
and a unique local solution $u \in \dot X[T,\infty]$, depending smoothly on $u_+$.
\end{theorem}

These theorems bring three key related improvements compared to Schottdorf's result, both of which play fundamental roles later on, in our proof of the main result in Theorem~\ref{thm:mNV}:
\begin{itemize}
    \item the transition from a model nonlinearity to the full nonlinearity in mNV, including the Beurling transform, see Remark~\ref{r:tobias}.
    \item the transition from small data to large data,
    \item the relaxed control norm $S^p$, and the corresponding continuation result relative to this norm. 
\end{itemize}

We note that Schottdorf's approach cannot be used to prove any large data result. The problem with the $U^2_{NV}$ norms in \eqref{tri-dual} is that they are not time-divisible, i.e., they cannot be made small by reducing the size of the time interval. 
To address this difficulty, here we will seek to replace one of the $U^2_{NV}$ norms by the Strichartz space $S^p$, which is used both as a source of smallness and as an extension criterion.

The spaces $S^p$ we use here are larger than $\dot X$. Precisely, we have the embeddings (see \eqref{dotXY} and \eqref{UV-to-Sp})
\[
\dot{X} \subset \dot{Y} \subset S^p
\]
so that 
\begin{equation}
	\label{SYX}
	\norm[S^p]{u} \lesssim \norm[\dot{Y}]{u} \lesssim \norm[\dot{X}]{u}.
\end{equation}

We will seek to replace the trilinear estimate by a stronger form. 
We denote 
by $NL(u_1,u_2,u_3)$ the symmetrized  map
\[ (u_1,u_2,u_3) \mapsto \int_0^t S_{NV}(t-s) N_{mNV}(u_1,u_2,u_3)(s) \, ds.
\]

With this notation, our main estimates read as follows:

\begin{lemma}
	\label{lemma:tri-S}
a)	For any $2 < p < \infty$ we have
	\begin{equation}
		\label{tri-S+}
     \begin{aligned}   
		\Bigl\| NL(u_1, u_2,u_3) \Bigr\|_{\dot{X}} 
		 \lesssim & \ 
				\norm[\ell^\infty S^p]{u_1} \norm[\dot{X}]{u_2} \norm[\dot{X}]{u_3} 
                +
				\norm[\dot{X}]{u_1}	\norm[\ell^\infty S^p]{u_2} \norm[\dot{X}]{u_3}
	\end{aligned}
    \end{equation}
b) For any $s \geq 0$ we have
	\begin{equation}
		\label{tri-S-s}
			\Bigl\| NL(u,u,u) \Bigr\|_{\dot{X^s}} 
		 \lesssim \\[5pt]
				\norm[\dot{X}]{u} \norm[\dot{X^s}]{u} \norm[\ell^\infty S^p]{u}.
	\end{equation}
\end{lemma}

\begin{remark}
We note that using two $\ell^\infty S^p$ norms in  \eqref{tri-S+}, rather than three, is needed in order to obtain estimates not only for solutions but also for  the linearized equation or more generally for differences, which require some symmetry breaking.
\end{remark}

\begin{remark}\label{r:ac}
The bounds in the lemma are stated on the 
real line, but clearly they can also be restricted to arbitrary time intervals $I$.
In this context, a straightforward observation
is that, if the space $\ell^\infty S^p$ is replaced by its slightly smaller subspace $\ell^\infty_0 S^p$, then we have the absolute continuity property 
\begin{equation}\label{AC-NL}
\lim_{|I| \to 0} \  \Bigl\| NL(u_1, u_2,u_3) \Bigr\|_{\dot{X}[I]} = 0.   
\end{equation}
This follows directly from the
similar property for the $\ell^\infty_0 S^p$ norm.
\end{remark}

\begin{proof}[Proof of Lemma~\ref{lemma:tri-S}]
a) By duality (see \eqref{tri} and \eqref{tri-dual}) it suffices
to consider the form
\[
I(u_1,u_2,u_3,v) = \int  N_{mNV}(u_1,u_2,u_3)  v \, dx \, dy \, dt
\]
and show that $I$ satisfies
the estimate
\begin{equation} \label{tri-Sb}
|I(u_1,u_2,u_3,v)| \lesssim 
(\norm[S^p]{u_1} \norm[\dot{X}]{u_2} \norm[\dot X]{u_3}+
				\norm[\dot{X}]{u_1}	\norm[S^p]{u_2} \norm[\dot{X}]{u_3})
\norm[\dot Y]{v}
\end{equation}
We bound $I$ through a Littlewood-Paley decomposition of $u_1$, $u_2$, $u_3$, $w$ into dyadic frequencies
		$[\lambda]:=\lambda_1$, $\lambda_2$, $\lambda_3$, $\lambda$,
\[
I(u_1,u_2,u_3,v) = \sum_{[\lambda]}
I(P_{\lambda_1} u_1,P_{\lambda_2}u_2,P_{\lambda_3} u_3,P_\lambda v)
:= \sum_{[\lambda]} I_{[\lambda]}
\]
We estimate each of the summands separately and then carry out the dyadic summation. 

Since in our estimates we do not 
need to distinguish between $u$ and $\bar u$ and between $B$ and $B^{-1}$, to fix the notations we only need to consider two types of terms, namely
\[
N^1_{mNV}(u) = \partial uB(u^2), \qquad
N^2_{mNV}(u) = uB(u \partial u).
\]
However, since $N_{nNV}$ is assumed to be symmetrized in our estimates, in order to dispense with the symmetrization we need to allow for 
an arbitrary placement of $u_1$, $u_2$ 
and $u_3$ in $N^1_{mNV}$, respectively $N^2_{mNV}$.

\medskip

Due to the presence of the Beurling operator 
we distinguish two cases:

\begin{enumerate}[label=\Alph*)]
\item  Either the pair of frequencies 
$\lambda_1, \lambda_2$ or the pair 
of frequencies $\lambda_3,\lambda$
is unbalanced. Then the Beurling operator (which can be shifted to either pair) can be localized to a fixed dyadic frequency, either 
$\max\{\lambda_1, \lambda_2\}$ or  $\max\{\lambda_3,\lambda\}$, in which case it has an integrable kernel and can be harmlessly discarded.

\item We have $\lambda_1 \approx \lambda_2$ and $\lambda_3 \approx \lambda$. Here we can no longer 
dispense with the Beurling operator, and instead we need to carry out the analysis at the level of mixed norm spaces. It is 
in this case that the lateral Strichartz spaces are needed.

\end{enumerate}
We now successively consider the two cases.

A. For each term in the dyadic sum we will 
choose one of the arguments to estimate in $S^p$, and the
others will be estimated in $\dot X$ or $\dot Y$ as needed.
The key principles are as follows:

\begin{itemize}
    \item In order for the expression $I_{\lambda_1,\lambda_2,\lambda_3,\lambda}$ to be nonzero, 
    the two highest frequencies must be comparable. Hence we organize the frequencies as
    \[
[\lambda] := \{ \lambda_1,\lambda_2,\lambda_3,\lambda \} = \{ \lambda_{min},\lambda_{med}, \lambda_{max},  \lambda_{max}\}.
    \]
 and correspondingly we denote
  \[
\{ P_{\lambda_1} u_1,P_{\lambda_2} u_2,P_{\lambda_3} u_3,P_\lambda v \} = \{ u_{min},u_{med}, u_{max},  u_{max}\}.
    \]
\item To prove \eqref{tri-Sb} we
may choose one of two entries to bound in $S^p$. In particular this implies that we can always avoid choosing 
$\lambda_{min}$, so we are left with two cases.

\item The derivative applied to $u_3$ yields a frequency factor, which we can always bound by $\lambda_{max}$.
\end{itemize}
Given the above discussion, we can simply assume that $I_{[\lambda]}$ has the form
\[
I_{[\lambda]} = \lambda_{max} \int 
P_{\lambda_1} u_1 P_{\lambda_2}u_2 P_{\lambda_3} u_3 P_\lambda v \, dx dy dt,
\]
and consider two cases:
\medskip

\emph{(i)} $u_{med}$ is estimated in $S_p$.
Using \eqref{bi-UU} to bound $u_{min} u_{max}$ in $L^2$, 
and \eqref{St.bi1} for the other two factors we arrive at
\begin{equation}
|I_{[\lambda]}| \lesssim \lambda_{min}^\frac12 \lambda_{mid}^{-\frac{1}{p}}\lambda_{max}^{-\frac{1}{r}}\| u_{min}\|_{U^2_{NV}}
\| u_{med}\|_{S^p} \| u_{max}\|_{S^r}
\| u_{max}\|_{U^2_{NV}}
\end{equation}
\medskip

\emph{(ii)} One $u_{max}$ is estimated in $S_p$.
Using \eqref{bi-UU} to bound $u_{min} u_{max}$ in $L^2$, 
and \eqref{St.bi1} for the other two factors we arrive at
\begin{equation}
|I_{[\lambda]}| \lesssim \lambda_{min}^\frac12 \lambda_{mid}^{-\frac{1}{r}}\lambda_{max}^{-\frac{1}{p}}\| u_{min}\|_{U^2_{NV}}
\| u_{med}\|_{S^r} \| u_{max}\|_{S^p}
\| u_{max}\|_{U^2_{NV}}
\end{equation}

\medskip

 In both cases  we remark that 
 \begin{itemize}
 \item 
we can bound the $S^r$ norm by the $V^2_{NV}$ norm, and 
 \item
 we can also use \eqref{bi-UV} instead to substitute either of the $U^2_{NV}$ norms
by $V^2_{NV}$ at the expense of a harmless factor 
of $1+ \log(\lambda_{max}/\lambda_{min})$.
\end{itemize}
\medskip

What is essential is that in both cases we have a gain 
$(\lambda_{min}/\lambda_{max})^c$ with $c > 0$; this will allow us to carry out the dyadic summation. The 
arguments are similar in all cases so we describe one
of them,
\[
\lambda_1 \leq \lambda_2 \leq \lambda_3 \approx \lambda.
\]
Here 
\[
\lambda_1 = \lambda_{min}, \qquad \lambda_2= \lambda_{med}, \qquad \lambda_3 \approx \lambda \approx \lambda_{max}.
\]
Then we use case (i) above with the $V^2_{NV}$ 
adjustment to estimate
\[
\begin{aligned}
\sum_{ \lambda_1 \leq \lambda_2 \leq \lambda_3 \approx \lambda} |I_{[\lambda]} |\lesssim & \ 
\sum _{ \lambda_1 \leq \lambda_2 \leq \lambda_3 \approx \lambda} \left(\frac{\lambda_{1}}{\lambda}\right)^c\| P_{\lambda_1} u_1\|_{U^2_{NV}}
\| P_{\lambda_2} u_{2}\|_{S^p} \| P_{\lambda_3} u_{3}\|_{S^r}
\| P_\lambda v\|_{V^2_{NV}}
\\ \lesssim & \ 
\| u_{2}\|_{\ell^\infty S^p} \| u_{3}\|_{\ell^\infty S^r}
\sum _{ \lambda_1 \leq \lambda} \left(\frac{\lambda_{1}}{\lambda}\right)^c\| P_{\lambda_1} u_1\|_{U^2_{NV}}
\| P_\lambda v\|_{V^2_{NV}}
\\ \lesssim & \ 
\| u_{2}\|_{\ell^\infty S^p} \| u_{3}\|_{\ell^\infty S^r}
\left(\sum_{\lambda_1} \| P_{\lambda_1} u_1\|_{U^2_{NV}}\right)^\frac12
\left(\sum_\lambda \| P_\lambda v\|_{V^2_{NV}}\right)^\frac12
\\ = & \ \| u_{2}\|_{\ell^\infty S^p} \| u_{3}\|_{\ell^\infty S^r}
 \| u_1\|_{\dot X}
 \| v\|_{\dot Y}
\end{aligned}
\]
as needed for \eqref{tri-Sb}.
All other cases are similar. 

\bigskip

B. Here we only need to consider 
dyadic terms where the two pairs of
frequencies on either side of the 
Beurling transform are equal. We denote the set of four frequencies
by $[\lambda] = \{\mu,\mu,\lambda,\lambda\}$, where
in order to fix the notations we assume that $\mu \lesssim \lambda$.
To reduce the number of cases we note
that the derivative which is present in $N_{mNV}$ will yield a factor of either $\lambda$ or $\mu$, so we assume the worst and replace it by a $\lambda$ factor. Then the integral to evaluate has the form
\begin{equation}\label{I-lm}
I_{\mu\lambda} = \lambda \int (P_{\mu} u_1 P_{\mu} u_2)B(P_{\lambda} u_3 P_\lambda v) \, dt dx dy
\end{equation}
but where we need to allow for arbitrary permutations of the functions $u_1,u_2,u_3$ and $v$.
For clarity we denote the expression above by $I_{\mu\lambda}^{1234}$, 
and its variants by permuting the 
indices $1234$ in the superscript.

 We evaluate the integral 
in two different ways. First, we 
use only regular Strichartz norms type norms, with $S^p$ on $u_1$ and $u_2$,
and also using the $L^p$ boundedness of the Beurling transform.
For  the expression $I_{\mu\lambda}^{1234}$
above, for instance, we obtain
\[
|I_{\mu\lambda}^{1234}|
\lesssim \mu^{-\sigma}
\lambda^{\sigma}
\| P_\mu u_1\|_{S^p} \|P_{\mu}u_2\|_{S^p}\|P_\lambda u_3\|_{S^r} \| P_\lambda v\|_{S^r}, \qquad \frac1p+\frac1r = \frac12, \quad \sigma = \frac2p,
\]
which has the desired norms but an unfavorable frequency balance unless $\mu \approx \lambda$. Other permutations yield different prefactors, but all are of the same form  
with $\sigma \in (0,1)$. This bound suffices 
in the balanced case $\mu \approx \lambda$.

If $\mu \ll \lambda$ then we need a second step,  where we use lateral Strichartz norms  and the boundedness of the Beurling transform in mixed norm spaces, see e.g. the book of Muscalu-Schlag~\cite{MS}, and also \cite{MiyachiTomita2018}.
To obtain a nonzero output the frequencies of $u_3$ and $v$ must be $\mu$ close, 
which implies that by using an angular partition of unity we can 
assume that $u_3$ and $v$ are frequency localized in the same conical region $C_e$ for some direction $e$. Then we obtain
\[
|I^{1234}_{\mu\lambda}|
\lesssim \mu^{-2+\frac{6}{p_1}}
\lambda^{-1+\frac{6}{r_1}}
\| P_\mu u_1\|_{S^{p_1}_e} \|P_{\mu}u_2\|_{S^{p_1}_e}\|P_\lambda u_3\|_{S^{r_1}_e} \| P_\lambda v\|_{S^{r_1}_e}, \quad \frac1{p_1}+\frac1{r_1} = \frac12.
\]
Here we choose $p_1$ close to $2$
and $r_1$ close to $\infty$ so that
\[
-2+\frac{6}{p_1} = 1 -\delta, 
\qquad -1+\frac{6}{r_1} = -1 +\delta
\]
with $\delta$ arbitrarily small.
Finally, we average the last to bounds, estimating all Strichartz norms but $S^p$ with the appropriate $U^2_{NV}$, respectively $V^2_{NV}$ norm. This gives
\[
|I_{\mu\lambda}^{1234}|
\lesssim \left(\frac{\mu}{\lambda}\right)^{\frac{1-\sigma -\delta}2}
\| P_\mu u_1\|_{S^p}^\frac12 \|P_{\mu}u_2\|_{S^p}^\frac12 
\| P_\mu u_1\|_{U^2_{NV}}^\frac12 \|P_{\mu}u_2\|_{U^2_{NV}}^\frac12 
\|P_\lambda u_3\|_{U^2_{NV}} \| P_\lambda v\|_{V^2_{NV}}.
\]
Here $\sigma < 1$, so by choosing $\delta$ small enough we can insure that the exponent of $\mu/\lambda$ 
is negative. This is exactly what is needed in order to insure the 
dyadic summation in $\mu$ and $\lambda$, which is simpler than in case A due to the restricted range of indices. This proves \eqref{tri-Sb}
for $I_{\mu\lambda}^{1234}$. The argument is similar for other permutations of the upper indices.

\medskip

b) The proof of \eqref{tri-S-s}
is nearly identical to the argument above. The quartic estimate to prove is now
\begin{equation} \label{tri-S-ss}
|I(u,u,u,v)| \lesssim 
\norm[\dot{X}]{u} \norm[\dot{X^s}]{u_2} \norm[S^p]{u}
\norm[\dot Y^{-s}]{v}.
\end{equation}
For this we use the same strategy as above, organizing the three $u$ factors so that
\begin{itemize}
\item $u_{max}$ is estimated in $\dot X^s$, 
\item $u_{min}$ is estimated in $\dot X$
\item the remaining $u_{med}$ is estimated in $S^p$.
\end{itemize}
This yields the exact same frequency balance as in case
(a), but with an additional 
$(\lambda/\lambda_{max})^s$ factor.
But this is always favorable, so the dyadic summation is identical to the one in case (a). 
\end{proof}

\begin{proof}[Proof of Theorem \ref{thm:mNV.global}]

We now use the estimate in the above Lemma in order to prove the large data local well-posedness. For this we still do a fixed point argument but in a weighted norm.

We rewrite our equation in the Duhamel form
\[
u(t) = S_{NV}(t) u_0 + \int_{0}^t S_{NV}(t-s) N(u(s)) \,  ds. 
\]
Suppose
\[
\|u_0\|_{L^2} = R.
\]
Then 
\[
\| S_{NV}(t) u_0 \|_{S^p} \lesssim \| S_{NV}(t) u_0\|_{\dot X} \lesssim R.
\]
We choose the time interval $I = [0,T]$ small enough 
so that 
\[
\| S_{NV}(t) u_0 \|_{S^p} \leq \epsilon.
\]
Denoting 
\[
v = u(t) - S_{NV}(t) u_0,
\]
we rewrite the above Duhamel form in terms of $v$,
\begin{equation}
v(t) =  \int_{0}^t S_{NV}(t-s) N(v(s) + S_{NV}(s) u_0) \, ds.
\end{equation}
We solve this using the contraction principle in $B_\delta(\dot X)$. By \eqref{tri-S+} we 
can estimate the worst term by 

\[
\| N(v,S_{NV}(s) u_0,S_{NV}(s) u_0) \|_{DU^2} \lesssim 
\|v\|_{\dot X} \|S_{NV}(s) u_0\|_{\dot X} \|S_{NV}(s) u_0)\|_{S^p}
\lesssim \epsilon R \|v\|_{\dot X},
\]
so we need
\[
\epsilon R \ll 1.
\]
Similarly 
\[
\| N(S_{NV}(s)u_0 ,S_{NV}(s) u_0,S_{NV}(s) u_0) \|_{DU^2} 
\lesssim \epsilon R^2,
\]
which gives us the choice 
\[
\delta = \epsilon R^2, 
\]
and which in turn should be $\ll R$.

We conclude that if $\epsilon \ll R^{-1}$ then we get a solution 
$v$ with 
\[
\|v\|_{U^2} \lesssim \epsilon R^2.
\]

\bigskip

Next we assume we know that we have a solution $u$ in a time interval 
$I=[0,T]$ with 
\[
\| u\|_{S^p(I)} \leq \epsilon.
\]
Then we apply the trilinear estimate in $I$ to get
\[
\| u\|_{\dot X(I)} \lesssim \|u_0\|_{L^2} + \epsilon \|u\|_{\dot X}^2
= R + \epsilon \|u\|_{\dot X(I)}^2.
\]
The same clearly holds in any subinterval 
$I_0 = [0,T_0] \subset I$. Since $S^p \subset \ell^\infty_0 S^p$, by the property \eqref{AC-NL} it follows that the restricted norm 
$\| u\|_{\dot X[I_0]}$ is a continuous function of $T_0$.

Assuming $\epsilon R \ll 1$, the above continuity property allows us to use a continuity argument with respect to $T_0$
in order to get
\[
\| u\|_{\dot X(I)} \lesssim R.
\]

\bigskip

Finally, suppose we have a solution $u$ in a time interval $I$ 
so that $u \in \dot X(I) \subset S^p(I)$. We divide $I$ into subintervals $I_j$
so that 
\[
\| u\|_{S^p(I_j)} \lesssim \epsilon \ll R^{-1}.
\]
We need about
\[
N = R^p \|u\|_{S^p}^p
\]
such intervals. Then in each such interval we have $\|u\|_{\dot X(I_j)} \lesssim R$
so adding up  we get 
\[
\|u\|_{\dot X(I)} \lesssim R^\frac12 .
\]
In particular  $u \in U^2_{NV}[I]$, which implies it has 
a limit in $L^2$ at time $T$, and thus it can be continued.

The proof of the final estimate in $\dot H^s$, namely \eqref{uSp-s}, is similar to the last argument but using part (b) of Lemma~\ref{lemma:tri-S} instead of part (a).
\end{proof}

\begin{proof}[Proof of Theorem \ref{thm:scattering}]

As in the previous proof, we rewrite our equation in the Duhamel form but with data from infinity
\[
u(t) = S_{NV}(t) u_+ - \int_t^\infty S_{NV}(t-s) N(u(s)) \,  ds. 
\]
Suppose
\[
\|u_+\|_{L^2} = R.
\]
Then 
\[
\| S_{NV}(t) u_+ \|_{S^p} \lesssim \| S_{NV}(t) u_+\|_{\dot X} \lesssim R.
\]
We choose the time interval $I = [T,\infty]$ small enough 
so that 
\[
\| S_{NV}(t) u_+ \|_{S^p[I]} \leq \epsilon.
\]
Denoting 
\[
v = u(t) - S_{NV}(t) u_+
\]
we rewrite the above Duhamel form in terms of $v$,
\begin{equation}
v(t) = - \int_t^\infty S_{NV}(t-s) N(v(s) + S_{NV}(s) u_0) \, ds.
\end{equation}
Finally we solve this using the contraction principle in $B_\delta(\dot X)$, exactly as in the proof of Theorem~\ref{thm:mNV.global}.
\end{proof}

\section{Nonlinear Gagliardo-Nirenberg Inequality}
\label{sec:gns}

In this section we prove novel bounds for 
the scattering transform 
\[
\calS: L^2(\R^2) \to L^2(\R^2)
\]
associated to the Davey-Stewartson system, and  studied in \cite{NRT2020}, see Theorem 1.2. 

We begin by recalling the definition of $\calS$.
Given a function $u \in L^2(\R^2)$ and $k \in\R^2\simeq\C$, let $m_\pm(z,k)$ be the solutions of the two equations
\begin{equation}\label{Jost}
\frac{\partial}{\partial\ol z} m_\pm =  \pm  e_{-k} u \ol
{m_\pm},
\end{equation}
with $m_\pm(\cdot,k)-1\in L^4(\R^2)$. (It is convenient to use the notation  $e_k(z) = e^{i(zk+\ol{zk})}$).
\
Existence and uniqueness of these Jost functions was established in  \cite{NRT2020} (Lemma 4.2), as a consequence of a pivotal result on
$\dbar$ equations with $L^2$ coefficients which we recall in Theorem ~\ref{t:Lq} below.
The scattering transform $\calS u$ of $u$ is then defined as:
\begin{equation}\label{defS}
\calS u (k) = \frac{1}{2\pi i}\int_{\R^2} e_k(z) \ol{u(z)}
\Big(m_+(z,k) + m_-(z,k)\Big) dz.
\end{equation}
In \cite{NRT2020} this was treated as a pseudodifferential operator with a non-smooth amplitude and shown to be bounded on $L^2$. The linearization of $\calS u$  at $u=0$ is essentially the Fourier transform:
\begin{equation}
\calS u (k)=\ol{\hat u(k)}+\mathcal O(u^2)
\end {equation}
where
\begin{equation}\label{FT}
\hat u (k) = \frac{i}{\pi}\int_{\R^2}e_{-k}(z)u(z)dz.
\end{equation}

The main properties of $\calS$ as a non-linear Fourier transform which were proved in \cite{NRT2020} have been summarized in 
 Theorem ~\ref{thm:NRT-S} of the Introduction.
 
In the setting of this paper, it would be ideal to have an 
$\dot W^{\frac14,4}$ bound for $\calS$, as in 
\eqref{W4}, generalizing the $L^4$ bound of \cite{NRT2020}. This remains an open problem for now. However, as it turns out, a slightly weaker version of \eqref{W4}, with slightly 
unbalanced choices of norms on the left and on the right, is still sufficient 
for the study of the modified Veselov-Novikov
problem in $L^2$.

Towards this goal, we consider two Strichartz pairs $(p,r)$ and $(p_1,r_1)$, i.e. so that 
\begin{equation}\label{2Str}
\frac{1}{p}+\frac{1}{r} = \frac{1}{p_1} + \frac{1}{r_1} = \frac12.
\end{equation}
These will be chosen so that the pair $(p,r)$
sits exactly in the middle between 
$(p_1,r_1)$ and $(\infty,2)$, where the latter corresponds to the energy norm.
This corresponds to
\begin{equation}\label{choose-r1}
\dfrac12 + \dfrac1{r_1} = \dfrac{2}r, \qquad 
\dfrac{1}{p_1}= \dfrac{2}{p}.
\end{equation}

For clarity, we will state two versions of 
our nonlinear GN inequality. The first one
is purely spatial, which is more natural in the context of the scattering transform  and is of independent interest as a complement to the nonlinear harmonic analysis theory in \cite{NRT2020}. The second one, which is a minor variation of the first, is its space-time counterpart. We begin with the spatial version, which is as follows:

\begin{theorem}\label{t:nonlinGNS}
Let $ 2 < r < 4$, 
 $\dfrac12 + \dfrac1{r_1} = \dfrac{2}r$ and $0<s<\dfrac12$.
Then the  scattering transform satisfies the fixed time 
bound 
\[
\|\mathcal  S u \|_{\dot B^{s,r}_{2}} \lesssim_{\|u\|_{L^2}} \|\hat u\|_{L^2}^\frac12 \| \hat u\|_{\dot B^{2s,r_1}_{2}}^\frac12
\]
\end{theorem}

We continue with the space-time version of the same result, which is stated directly in the context of the $S^p$ spaces:

\begin{theorem}\label{t:xtGNS} 
Let $(p,r)$ and $(p_1,r_1)$ be Strichartz pairs 
as in \eqref{2Str}, \eqref{choose-r1}
so that  $ 2 < r < 4$. Then the  scattering transform satisfies the space-time  bound
\begin{equation}
    \label{nonlinGNS.2}
        \|\mathcal S u \|_{\ell^2 L^{p} \dot W^{\frac1{p},r}} \lesssim _{\|u\|_{L^2}}\| \hat u \|_{\ell^2 L^\infty L^2}^\frac12 \| \hat u \|_{\ell^2 L^{p_1} \dot W^{\frac1{p_1},r_1}}^\frac12
\end{equation}

\end{theorem}
Using the $S^p$ notation, the above estimate
reads
\[
\| \mathcal S u \|_{\ell^2 S^p}  \lesssim_{\|u\|_{L^2}}
 \|\hat u\|_{\ell^2 L^\infty L^2}^\frac12 \| \hat u\|_{\ell^2 S^{p_1}}^\frac12
\]
which suffices for the global result for the mNV system.

To prove Theorem \ref{t:nonlinGNS} we look at variations of $Su$  
\begin{equation}
 \mathcal Su_1 - \mathcal Su_2 =   T_{u_1,u_2} (u_1-u_2).
\end{equation}
Here $T_{u_1,u_2}$ denote the following pseudodifferential operators 
introduced in \cite{NRT2020}:
\begin{equation}\label{T_define}
T_{u_1,u_2} f (k)  = -\frac{i}{\pi}\Big(  \int  e_k(z)
\ol{f(z)}a(z,k)\dnuz -  \int
e_k(z){f(z)}b(z,k)\dnuz\Big),
\end{equation}
with
\begin {equation}
a(z,k) = \ol{m^1_{\ol{u_2}}(z,-k)}m^1_{u_1}(z,k)\\
\end{equation}
and
\begin{equation}
b(z,k) = \ol{m^2_{\ol{u_2}}(z,-k)}m^2_{u_1}(z,k).
\end{equation}

In particular,the differential of $S$
is given by 
\begin{equation}
DS(u) = T_{u,u}.    
\end{equation}
As proved in \cite{NRT2020}, the operators 
$T_{u_1,u_2}$ are bounded in $L^2$, precisely we have
\begin{equation}\label{DS-L2}
\|  T_{u_1,u_2}  f\|_{L^2} \lesssim C(\|u_1\|_{L^2},\|u_2\|_{L^2}) \|f\|_{L^2} 
\end{equation}

One may ask whether these operators also satisfy  the $L^p$ bounds
\begin{equation}\label{DS-Lr}
\|  T_{u_1,u_2}  f\|_{L^p} \lesssim C(\|u_1\|_{L^2},\|u_2\|_{L^2}) \|\hat f\|_{L^p} 
\end{equation}
for $p \neq 2$. We do not know the answer to this 
question, however we can prove an interpolated version,
which will suffice for the proof of Theorem~\ref{t:nonlinGNS}.

\begin{proposition}\label{p:Lp-interp}
Assume that $u_1,u_2 \in L^2$. Then with $r$ and $r_1$ as in Theorem~\ref{t:nonlinGNS} the following estimate holds:
\begin{equation}\label{Tq-Lr}
 \| T_{u_1,u_2} f\|_{L^{r}} \lesssim  C(\|u_1\|_{L^2},\|u_2\|_{L^2})  \| f\|_{L^2}^\frac12 \| \hat f\|_{L^{r_1}}^\frac12   
\end{equation}
\end{proposition}

\begin{proof}
The proof of the proposition requires several steps, extending corresponding results in \cite{NRT2020}.

\subsection{ A \texorpdfstring{$\dbar$}{DBAR} bound}
An important role in \cite{NRT2020}
is played by the operator 
\begin{equation}\label{Lu}
   L_u w = \dbar u + u \bar w. 
\end{equation}
A key result in \cite{NRT2020} asserts that this operator is invertible in a symmetric setting:

\begin{theorem}\label{t:Lq}
Let $u \in L^2$.
Then we have
\begin{equation}
\| L_u^{-1} f \|_{\dot H^\frac12} \lesssim C(\|u\|_{L^2}) \| f\|_{\dot H^{-\frac12}}.   
\end{equation}
\end{theorem}

Here we need to prove an unbalanced extension of this result, namely

\begin{theorem}\label{t:Lq-s}
Let $u \in L^2$.
Then for any $0 < s < 1$ we have
\begin{equation}
\| L_u^{-1} f \|_{\dot H^s} \lesssim C(\|u\|_{L^2}) \| f\|_{\dot H^{s-1}}.  
\end{equation}
\end{theorem}

\begin{proof}
We will use the special case $s = \frac12$ proved in \cite{NRT2020} 
to prove the more general result. By duality it suffices to 
consider the case $s > \frac12$. 

A naive argument would be to try to apply the $s = \frac12$ case to the function $v = A^\sigma w$ where 
$A = |D|$,  $w = L_u^{-1} f$ and $\sigma = s -\frac12$. The function $v$ solves
\[
L_u v = A^\sigma f + R(u, \bar v),
\]
where the error term $R$ has the form
\[
R(u, v) =  (q- A^\sigma q A^{-\sigma}) v
\]
The argument would be concluded if we knew that 
\begin{equation}\label{good-R}
\| R(u,v) \|_{\dot H^{-\frac12}} \ll_{\|u\|_{L^2}} \|v\|_{\dot H^\frac12}
\end{equation}
But instead we only have 
\[
\| R(u,v) \|_{\dot H^{-\frac12}} \lesssim \|u\|_{L^2} \|v\|_{\dot H^\frac12}
\]
without smallness. The key to gain smallness is to 
first observe that the above inequality can be improved
to a Besov version
\begin{equation}\label{R-Besov}
\| R(u,v) \|_{\dot H^{-\frac12}} \lesssim \|u\|_{B^{0,2}_{\infty}} \|v\|_{\dot H^\frac12}
\end{equation}
This is proved using a standard Littlewood-Paley trichotomy,
and noting that there is an off-diagonal gain. 

Here we already win if $\|u\|_{B^{0,2}_{\infty}} \ll 1$.
But even if this is not the case, then the obstruction comes 
from finitely many dyadic frequencies of $u$, call them 
$j_1, \cdot, j_N$ with $N = N(\|u\|_{L^2})$. Then we decompose $u$ into a good and a bad portion,
\[
u = u_g +u_b
\]
where 
\begin{equation}
\| u_b\|_{B^{0,2}_{\infty}} \ll_{\|u\|_{L^2}}  1
\end{equation}
and 
\begin{equation}
 u_b = \sum_{k=1}^N u_{j_k}   
\end{equation}

To avoid the above difficulty, we modify the operator $A$ above in order  to avoid large contributions from these frequencies. Precisely, we redefine $A$ to be the multiplier $A(|D|)$ where its symbol is chosen to satisfy $a(0) = 0$ and
\[
a'(\xi) =
\left\{
\begin{aligned}
& 0 \qquad  \xi \in \bigcup \,  [2^{j_k -m},2^{j_k +m}]
\\
& 1 \qquad \text{otherwise}    
\end{aligned}
\right.
\]
where the parameter $m$ is chosen large enough. This also 
changes the expression for $R$ above.

We now need two ingredients:

\bigskip

(i) The estimate \eqref{R-Besov} remains valid, 
uniformly with respect to our choice of $A$ above.
This is a consequence of the uniform bound
\begin{equation}
1 \leq \frac{a(\xi)}{a(\eta)} \leq \frac{\xi}{\eta}, \qquad 0 < \eta < \xi.    
\end{equation}
This allows us to treat the contribution of $u_{g}$ perturbatively.

\bigskip

(ii) The contribution of $u_b$ is now small,
\begin{equation}
\| R(u_b,v) \|_{\dot H^{-\frac12}} \lesssim 2^{-\frac{m}2} 
\| u\|_{L^2} \|v\|_{\dot H^\frac12}
\end{equation}
This is due to the fact that the only nonzero contributions in $R$  occur only when the $u_b$ and $v$ inputs are $m$ separated. 

\bigskip

The two ingredients above allow us to choose $m$ large enough, depending only on $\|u\|_{L^2}$, so that 
\eqref{good-R} holds. This in turn implies that we can treat $R$ perturbatively, and obtain the bound
\[
\| A^{\sigma} w\|_{\dot H^\frac12} \lesssim \| A^\sigma f\|_{\dot H^{-\frac12}}
\]
To obtain the conclusion of the theorem, it suffices to observe that for the symbol $A$ we have the bound
\[
c(\|u\|_{L^2}) \xi \leq a(\xi) \leq \xi.
\]

\subsection{{Estimates on the Jost functions}}

\begin{lemma}
\label{l:dbar}
 For $u \in L^2$ and $4\leq p < \infty$
 we have
 \begin{equation}\label{dbar-}
 \| \dbar^{-1} e_{-k} u\|_{L^p}  \lesssim 
 \|u\|_{L^2}^{1-\frac2p} (M\hat u(k))^{\frac2p} 
 \end{equation}
\end{lemma}
\begin{proof}
This is obtained by interpolating between 
\begin{equation}
 \| \dbar^{-1} e_{-k} u\|_{L^4}  \lesssim 
 \|u\|_{L^2}^{\frac12} (M\hat u(k))^{\frac12} 
 \end{equation}
 which is Corollary 2.2 in \cite{NRT2020},
 and 
 \[
 \| \dbar^{-1} e_{-k} u\|_{VMO}  \lesssim 
 \|u\|_{L^2}
\]
which is immediate since $\dot H^1 \subset VMO$.
\end{proof}

\begin{lemma}\label{l:u-Lp}
Suppose that $u,f \in L^2$. Then for every $k$ 
with $M \hat f(k) < \infty$ there is an unique solution 
$w(\cdot,k) \in L^4$ to
\begin{equation}
 L_{e_{-k} u} w = e_{-k} f  
\end{equation}
Moreover, for any $4 \leq p < \infty$ we have
\begin{equation}
\|  w(\cdot,k) \|_{L^p} \lesssim_{\|u\|_{L^2}} 
\|f\|_{L^2}^{1-\frac2p} \|M\hat f(k)^\frac2p
\end{equation}
\end{lemma}

\begin{proof}
The first part is Lemma~4.1 in \cite{NRT2020}. For the second part, we peel off the leading term of $w$ and denote
\[
v = w - \dbar^{-1} (e_{-k} f)
\]
This solves
\[
L_{e_{-k u}} w = - e_{-k} u \partial^{-1}(e_{k} \bar f)
\]
Using Theorem~\ref{t:Lq-s} for $s = 1-\dfrac2p$ and Sobolev embeddings we have
\[
\|v\|_{L^p} \lesssim \|v\|_{H^s} \lesssim_{\|u\|_{L^2}}
\| u \partial^{-1}(e_{k} \bar f)\|_{\dot H^{s-1}}
\lesssim _{\|u\|_{L^2}}
\| u \partial^{-1}(e_{k} \bar f)\|_{L^{p^*}}
\]
with
\[
\frac{1}{p^*} = \frac12 + \frac1p
\]
Then we obtain
\[
\|w\|_{L^p} \lesssim _{\|u\|_{L^2}}
\|\partial^{-1}(e_{k} \bar f)\|_{L^{p}}
\]
and conclude using Lemma~\ref{l:dbar}.

\end{proof}

\begin{lemma}\label{l:m}
 Suppose $u \in L^2$, $M\hat u(k) < \infty$ and $4 \leq p < \infty$. Then for the Jost solutions constructed in Lemma~4.2 of \cite{NRT2020} we have
\begin{equation}\label{ma}
 \| m^\pm(\cdot,k)-1\|_{L^p} + \|  m^1(\cdot,k)-1\|_{L^p}
 +\|m^2(\cdot, k)\|_{L^p} 
 \lesssim_{\|u\|_{L^2}} (M\hat u(k))^\frac2p
\end{equation}
and
\begin{equation}\label{mb}
 \| \dbar m^1(\cdot,k)-1\|_{L^{p^*}} 
 \lesssim_{\|u\|_{L^2}} (M\hat u(k))^\frac2p
\end{equation}
 
\end{lemma}
\begin{proof}
    
\end{proof}
The functions $r^{\pm} = m^{\pm}-1$ solve
\[
L_{\mp e_{-k} u} r_{\pm} = \pm e_{-k} u.
\]

Then \eqref{ma} follows from Lemma~\ref{l:u-Lp}, and 
\eqref{mb} follows from \eqref{ma} and the relation
\[
\dbar m_1 = um_2.
\]
    
\end{proof}

\subsection{ The interpolated \texorpdfstring{$L^p$}{Lp} bound}

Here we prove Proposition~\ref{p:Lp-interp}. For this we recall from \cite{NRT2020} 
 the decomposition of the operator $T_{u_1,u_2}$ into five components,
\begin{align*}
T_{u_1,u_2} f  = -\frac{i}{\pi}\Big(P_1{f}(k) +
P_2{f}(k) + P_3{f}(k) + P_4{f}(k)
+ P_5{f}(k)\Big)
\end{align*}
where 
\begin{align*}
P_1{f}(k) &= \int e_k\ol{f}(\ol{
m^1_{\ol{u_2}}}-1)({m^1_{u_1}}-1)dz\\
P_2{f}(k) &=\int e_k\ol{f}(\ol{ m^1_{\ol{u_2}}}-1)dz\\
P_3{f}(k) &=\int e_k\ol{f}({m^1_{u_1}}-1)dz\\
P_4{f}(k) &=\int e_k\ol{f}dz\\
P_5{f}(k) &=-\int e_{-k}
{f}\ol{m^2_{\ol{u_2}}}{m^2_{u_1}}dz.
\end{align*}
We will separately estimate these components, in increasing order of difficulty:
\bigskip

\textbf{a) The bound for $P_4$:}
This is done by interpolation,
\[
\| P_4 f\|_{L^r} \leq \| \hat f\|_{L^2}^\frac12
\| \hat f\|_{L^{r_1}}^\frac12
\]

\bigskip

\textbf{b) The bounds for $P_3,P_2$:}
Here we use H\"older's inequality followed by Lemma~\ref{l:dbar} and 
Lemma~\ref{l:m} to estimate
\[
|P_3 f(k)| \lesssim \| \dbar^{-1} (e_k f)\|_{L^4}
\| \dbar m^1_{u_1} \|_{L^\frac43}
\lesssim_{\|u\|_{L^2}} \|f\|_{L^2}^\frac12 ( M\hat f(k))^\frac12(M\hat u(k))^\frac12
\]
which by H\"older's inequality in $k$ yields again 
\[
\| P_3 f\|_{L^r} \lesssim_{\|u\|_{L^2}} \| f\|_{L^2}^\frac12
\| \hat f\|_{L^{r_1}}^\frac12
\]
The bound for $P_2$ is similar.

\bigskip

\textbf{c) The bound for $P_1$:}
After one integration by parts we have 
\[
|P_1{f}(k)| \lesssim \int |\dbar^{-1}(e_k\ol{f})|
(|\dbar \ol{m^1_{{u_2}}}||{m^1_{u_1}}-1|
+|m^1_{\ol{u_2}}-1|\dbar {m^1_{u_1}}|
dz
\]
Estimating the three types of factors via \eqref{dbar-}, \eqref{ma} and \eqref{mb} we obtain 
\[
|P_1{f}(k)| \lesssim_{\|u_{1,2}\|_{L^2}} 
\|f\|_{L^2}^{1-\frac2{p_0}} \|(M\hat f(k))^{\frac2{p_0}}
|(M\hat u_1(k))^{\frac2{p_1}}|(M\hat u_2(k))^{\frac2{p_2}}
\]
provided that
\[
\frac{1}{p_0}+\frac1{p_1}+\frac{1}p_2 = \frac12, \qquad 4 \leq p_j < \infty. 
\]
 Here we choose $p_0 = 4$, $p_1=p_2=8$ and apply H\"older's inequality to arrive at
 \[
\|P_1\|_{L^r} \lesssim \|f\|_{L^2}^\frac12 
\| \hat f\|_{L^{r_1}}^\frac12 \| u_1\|_{L^2}^\frac14 
\| u_2\|_{L^2}^\frac14. 
 \]

\bigskip

\textbf{d) The bound for $P_5$:} Here we recall that $m^2$ 
solves
\[
\partial (e_k m^2) = e_k \bar u m^1
\]
Then we rewrite $P_5 f$ in the form 
\[
\begin{aligned}
P_5{f}(k) &=-\int e_{-k}
{f}\, \ol{ e_k m^2_{\ol{u_2}}}\, {e_k m^2_{u_1}}dz
\\
 &=-\int \dbar^{-1}(e_{-k}{f}) 
(\ol{ \partial (e_k m^2_{\ol{u_2}})}{e_k m^2_{u_1}}+
\ol{ e_k m^2_{\ol{u_2}}}{\dbar(e_k m^2_{u_1}}))
dz
\\
 &=-\int \dbar^{-1}(e_{-k}{f}) 
(  q \bar m^1_{\ol{u_2}}  m^2_{u_1}+
\ol{ e_z m^2_{\ol{u_2}}}{(\partial^{-1} \dbar)(e_z \bar q m^1_{u_1}}))
dz
\end{aligned} 
\]
Here we split $m^1 = (m^1-1)+1$ and use the $L^p$ boundedness
of the Beurling transform and H\"older's inequality to estimate
\[
|P_5{f}(k)| \lesssim \|\dbar^{-1} (e_{-k} f)\|_{L^4}
\| u_{1,2}\|_{L^2} (\|  m^2_{u_1}\|_{L^4}
+ \|  m^1_{{u_2}}\|_{L^8}  \|m^2_{u_1}\|_{L^8}
+ \|  m^2_{{u_2}}\|_{L^4} + \|m^2_{{u_2}}\|_{L^8} \| m^1_{u_1}\|_{L^8})
\]
Now we apply \eqref{ma} and \eqref{dbar-} to arrive at
\[
|P_5{f}(k)| \lesssim \| f\|_{L^2}^\frac12 (M\hat f(k))^\frac12 ( (M\hat u_1(k))^\frac12+ (M\hat u_2(k))^\frac12)
\]
Finally, by H\"older's inequality we arrive at
\[
\| P_5 f\|_{L^r} \lesssim_{\|u_{1,2}\|_{L^2}} \|f\|_{L^2}^\frac12 \| \hat f\|_{L^{r_1}}^\frac12
\]
as needed.

\subsection{ The proof of the GN inequality}

Here we use Proposition~\ref{p:Lp-interp} in order to prove 
Theorem~\ref{t:nonlinGNS}.

To estimate $\calS u$ we  consider the sequence (continuum) of data 
\[
u_{j} = \phi( 2^{-j}z) u
\]
where $\phi$ is a smooth test function which equals $1$ around $z = 0$. Then 
\begin{equation}\label{telescope}
\calS u = \lim_{j \to \infty} \calS u_j = \sum_{j \in \Z} \calS u_{j+1} - \calS u_j
\end{equation}
This is akin to a nonlinear paradifferential expansion for the nonlinear operator $S$ as a function of $\hat u$.

We will estimate the terms in the sum in $L^r$; we contend that they are essentially localized at frequency $2^j$. To prove this we consider first higher frequencies
and then lower frequencies.

For higher frequencies we neglect the difference and bound each term separately, using the linearization of $S$.
We have the Galilean symmetry
\[
S (e_h u) = Su(\cdot+h)
\]
therefore, differentiating in $h$, we get
\[
\nabla_k Su = T_{u,u} (i z q)
\]
Applying  Proposition~\ref{p:Lp-interp}
with $u=u_j$ we obtain
\begin{equation}\label{e-high}
  \| \nabla_k Su_j\|_{L^r} \lesssim_{\|u\|_{L^2} } \|\partial_k \hat u_j\|_{L^2}^\frac12 \| \partial_k \hat u_j\|_{L^{r_1}}^\frac12
\end{equation}

On the other hand for differences we directly 
apply Proposition~\ref{p:Lp-interp} 
and estimate
\begin{equation}\label{e-low}
\| S u_{j+1} - S u_j \|_{L^r} \lesssim _{\|u\|_{L^2} }
\|  u_{j+1}-u_j \|_{L^2}^\frac12 
\|  \hat u_{j+1} - \hat u_j \|_{L^{r_1}}^\frac12 
\end{equation}

To estimate $Su$ in our desired Besov space we interpret the telescopic sum in \eqref{telescope} as a Littlewood-Paley decomposition with tails, and write
\[
\begin{aligned}
\| Su\|_{B^{s,r}_{2}}^2 
\lesssim & \  \sum_j 2^{2js} \| Su_{j+1} -Su_j\|_{L^r}^2+ 2^{2j(s-1)}\| \nabla(Su_{j+1} -Su_j)\|_{L^r}^2
\\
\lesssim & \sum_j  2^{2j s}
\|  u_{j+1}-u_j \|_{L^2}
\|  \hat u_{j+1} - \hat u_j \|_{L^{r_1}} 
+ 2^{2j(s-1)}\|\partial_k \hat u_j\|_{L^2} \| \partial_k \hat u_j\|_{L^{r_1}}
\\
\lesssim & \left(\sum_j 
  \|  u_{j+1}-u_j \|_{L^2}^2
  + 2^{-2j} \|\partial_k \hat u_j\|_{L^2}^2\right)^\frac12\\
  &\quad \times 
\left( \sum_j  2^{4js}
\|  \hat u_{j+1} - \hat u_j \|_{L^{r_1}}^2+
2^{2j(2s-1)}\| \partial_k \hat u_j\|_{L^{r_1}}^2
 \right)^\frac12
\\
\approx & \  \| u\|_{L^2} \| \hat u\|_{B^{2s,r}_{2}} 
\end{aligned}
\]
where the very last bound requires $0 < s < \frac12$.
Allowing tails here (first line in the computations above) corresponds to the fact that 
$Su_{j+1}-Su_j$ are not exactly localized at frequency $2^j$, but instead have tails at other frequencies, which however decay. This can be captured
by an inequality of the form 
\[
\| P_k f_j\|_{B^{s,r}_{2}} 
\lesssim  2^{-c|j-k|} (2^{js} \| f_j\|_{L^r} + 2^{j(s-1)}\| \nabla f_j\|_{L^r})
\]
for appropriately chosen positive $c$.

The first line of the previous inequality relies on these bounds applied to $f_j = Su_{j+1}-Su_j$.
For this we have 
\[
\| P_k \sum f_j\|_{B^{s,r}_{2}} 
\lesssim  2^{-c|j-k|} (2^{js} \| f_j\|_{L^r} + 2^{j(s-1)}\| \nabla f_j\|_{L^r})
\]
which we need to square sum. On the right we have a
discrete convolution with the kernel $2^{-c|l|}$ which is summable, so the first line of the previous inequality directly follows.

\end{proof}

\subsection{The space-time version of GN}

In this section we prove the space-time version of the nonlinear GN inequality,
namely Theorem~\ref{t:xtGNS}.
The proof is a minor variation of the fixed time bound; precisely,
the proof forks from the proof of Theorem~\ref{t:nonlinGNS} at the very end. From the fixed time bounds \eqref{e-high} and \eqref{e-low} we 
directly get by H\"older's inequality 
\begin{equation}\label{e-high-xt}
  \| \nabla_k Su_j\|_{L^p L^r} \lesssim_{\|u\|_{L^2} } \|\partial_k \hat u_j\|_{L^\infty L^2}^\frac12 \| \partial_k \hat u_j\|_{L^{q_1}L^{r_1}}^\frac12
\end{equation}
\begin{equation}\label{e-low-xt}
\| S u_{j+1} - S u_j \|_{L^q L^r} \lesssim _{\|u\|_{L^2} }
\|  u_{j+1}-u_j \|_{L^\infty L^2}^\frac12 
\|  \hat u_{j+1} - \hat u_j \|_{L^{q_1} L^{r_1}}^\frac12 
\end{equation}
and then complete the proof exactly as for Theorem~\ref{t:nonlinGNS}.

\section{The Modified Novikov-Veselov Equation}
\label{sec:gwp}

The goal of this section is to prove Theorem~\ref{thm:mNV}. Our starting point is provided by the solutions to the mNV equation obtained by inverse scattering. If $\calS$ denotes the scattering map, it was shown in \cite[Proposition 5.1]{Perry2014} that if $u_0 \in S(\R^2)$, then the function 
\begin{equation}
    \label{mNV.sol.IS}
    u(t) = \calS \left(e^{-it((\diamond)^3+(\overline{\diamond})^3)} \calS (u_0)(\diamond) \right)
\end{equation}
gives a global, classical solution of the mNV equation
which is uniformly bounded in $L^2$ and belongs to $\calS$ at each time.\footnote{The difference in sign in the exponential here is due to a different normalization of the scattering transform, following \cite{NRT2020} rather than \cite{Perry2014}.}

Applying the linear Strichartz estimates to the linear NV
flow we obtain
\begin{equation}
  \| \calF\left(e^{-it((\diamond)^3+(\overline{\diamond})^3)} \calS (u_0)(\diamond)  \right)\|_{\ell^2 L^{p_1} \dot W^{\frac1{p_1},r_1}} \lesssim \| \calS (u_0)\|_{L^2}
  = \|u_0\|_{L^2}
\end{equation}
Then by the nonlinear GN bound in Theorem~\ref{t:xtGNS}
we obtain the global in time estimate

\begin{equation}\label{global-se}
\| u\|_{\ell^2 L^{p} \dot W^{\frac1{p},r}} \lesssim_{\|u_0\|_{L^2}}
  \|u_0\|_{L^2}   
\end{equation}

Now we consider $L^2$ initial data. For this we have 
two distinct strategies to obtain solutions:

\begin{enumerate}[label=(\alph*)]
\item \emph{By density:} Since the Schwartz space $S$ is dense in $L^2$ and the scattering transform $\calS$
is smooth in $L^2$, the solution operator defined by \eqref{mNV.sol.IS} extends to a smooth map 
\[
L^2 \ni u_0 \to u \in C(\R; L^2). 
\]

\item \emph{By local well-posedness:} Given any $L^2$ 
initial data $u_0$, by Theorem~\ref{thm:mNV.global} there exists a small time $T > 0$ so that for each 
$\tilde u_0$ near $U_0$ there exists a unique solution $\tilde u \in U^2_{NV}[0,T;L^2]$, depending smoothly on $u_0$. Further
this result extends for as long as the 
$\ell^2 L^{p}[0,T; \dot W^{\frac1{p},r}]
$ remains finite.
\end{enumerate}

We contend that the two solutions are one and the same, and that, in particular, the solution provided by (b) is global in time.

Suppose $u_0$ in $L^2$ and denote
\[
\|u_0\|_{L^2}^2 = E,
\]
on which the implicit constants will depend in all estimates that follow.

Let $u^{n}_0 \in S$ 
so that $u^n_0 \to u_0$ in $L^2$. Let $T_{max} \in (0,\infty]$
the maximal time up to which the solution $u$ in (b)
associated to the initial data $u_0$ has a finite $\ell^2 L^{p} \dot W^{\frac1{p},r}$ norm.
For $T < T_{max}$, both the solutions provided in (a) and in (b) are the unique limits of the Schwartz solutions $u^{n}$ in $[0,T]$, so they must agree. 
Furthermore, by (b) we have the convergence 
\[
u^n \to u \qquad \text{in} \ U^2_{NV}[0,T;L^2]
\]
This in turn implies convergence in $\ell^2 L^{p}[0,T; \dot W^{\frac1{p},r}]$, which  by \eqref{global-se}
implies the estimate
\begin{equation}
  \| u\|_{\ell^2 L^{p}[0,T; \dot W^{\frac1{p},r}]} \lesssim_E
  \|u_0\|_{L^2}    
\end{equation}
Here it is essential the the implicit constant does not depend on $T$ at all. Passing to the limit $T \to T_{max}$, we arrive at the estimate 
\begin{equation}
  \| u\|_{\ell^2 L^{p}[0,T_{max}; \dot W^{\frac1{p},r}]} \lesssim_E
  \|u_0\|_{L^2}    
\end{equation}

If $T_{max} < \infty$, this bound, by Theorem~\ref{thm:mNV.global}, implies that the local solution in (b) extends beyond $T_{max}$, which contradicts the maximality of $T_{max}$. Hence 
we must have $T_{max} = \infty$, which implies 
that the solutions defined by (a) and (b) agree globally, and that the bound \eqref{global-se} hold for all $L^2$ solutions. By the estimate \eqref{uSp},
this shows that the global solutions satisfy
\begin{equation}
\|u\|_{\dot X} \lesssim_E \|u_0\|_{L^2}    
\end{equation}
globally in time. 

We now argue that the data to solution map 
\[
L^2 \ni u_0 \to u \in \dot X
\]
is smooth. Given $u_0 \in L^2$ and $\epsilon > 0$, we consider
a partition of the time axis into finitely many subintervals $I_j$ (of which two are unbounded)
so that 
\[
\| u \|_{\ell^2 L^{p}[I_j; \dot W^{\frac1{p},r}]} \leq \epsilon.
\]
If $\epsilon$ is small enough, depending only on $E$,
then the proof of Theorem~\ref{thm:mNV.global}
shows that the data to solution map is smooth in a neighbourhood of the solution $u$. Reiterating, the desired smoothness follows globally in time, concluding the 
proof of part (iii) of the theorem.

The $H^s$ bound in part (iv) of the theorem is a direct consequence of part (c) of Theorem~\ref{thm:mNV.global}.

To prove that the solutions scatter, we recall that for $u \in \dot X$ the limits
\begin{equation}\label{asympt}
u_{\pm} = \lim_{t \to \pm \infty} S_{NV}(-t) u(t)  
\end{equation}
exist in $L^2$. By the prior argument, these limits 
are smooth as functions of $u_0 \in L^2$.  
The formula 
\begin{equation}
u_{\pm} = \widecheck{\calS(u_0)}
\end{equation}
for these limits is proved by an argument similar to that used in \cite{NRT2020} to obtain the corresponding expression for DS-II. 

It then follows from the properties of the scattering transform 
$\calS$ (see Theorem~\ref{thm:NRT-S}) that the resulting wave operators are diffeomorphisms of $L^2$. In particular, these wave operators are complete, i.e. for each $u_{+} \in L^2$ there exists a unique global $L^2$ solution $u \in \dot X$, depending smoothly on $u_+$, so that \eqref{asympt} holds, and similarly for $u_-$.

Finally, the pointwise bound \eqref{ptws-bound} follows from \eqref{mNV.sol.IS} and the pointwise bound  \eqref{point-calS} on the scattering transform $\calS$ (as part of a more general class of pseudodifferential operators) established in \cite{NRT2020}.

This concludes the proof of Theorem~\ref{thm:mNV}.

\section{The Novikov-Veselov Equation}
\label{sec:nv}

In this section, we will study the Cauchy problem for the Novikov-Veselov equation with initial data in the range of the Miura map. We recall the Cauchy problem for the NV equation:
\begin{equation}
	\label{NV.Cauchy}
	\left\{\begin{aligned}
	& q_t + \diff^3 q + \dbar^3 q - \frac34(q \dbar^{-1} \diff q)-\frac34 (q \diff^{-1} \dbar q) = 0\\
		& \left. q \right|_{t=0} = q_0
	\end{aligned}	\right.
\end{equation}
for a real valued function $q$. 

We recall the connection between the mNV and NV  equations via the Miura map 
\begin{equation}
	\label{mNV.Miura}
	\calM(u)= 2 \diff u + |u|^2,
\end{equation}
where the domain of $\calM$ is understood to consist of functions $u$ in $L^2$ which satisfy the algebraic condition
\begin{equation}
\label{u-constraint.bis}
\diff u = \overline{\diff u},
\end{equation}
This has been rigorously verified at the level of more regular solutions, see for instance the work \cite{Perry2014},  which  works with functions $u \in H^{2,1}(\R^2) \cap L^1(\R^2)$ satisfying both \eqref{u-constraint.bis} and 
\[
\int_{\R^2} u \, dx \, dy =0.
\]

In our case we have the mNV global well-posedness result at the $L^2$ level, 
which corresponds to interpreting the  Miura map as a smooth map
\begin{equation}
L^2 \ni u \Longrightarrow q = \calM(u) \in \dot H^{-1} + L^1.    
\end{equation}
This naturally leads to the study of the NV flow with initial data $q_0 \in 
\dot H^{-1} + L^1$. 

The class of functions $u \in L^2$ satisfying the constraint \eqref{u-constraint.bis} is easily seen to be invariant with respect to the mNV
flow, so applying the Miura map to a solution to the mNV problem at the above regularity will heuristically yield solutions for the NV problem. 
This in principle yields a well-posedness result for the 
NV equation, but with an important caveat, namely that 
this well-posedness result would be restricted to the range of the Miura map.  
The two key questions that need to be answered in this context are
\begin{enumerate}[label=(\roman*)]
\item Study the invertibility of the Miura map.
\item Describe the range of the Miura map as a subset of $\dot H^{-1} + L^1$.
\end{enumerate}
In the next subsection we will show that, while not surjective, the Miura map has a continuous inverse. In Section~\ref{sec:spectral} we fully describe the range of the Miura map as the set 
of potentials $q \in \dot H^{-1} + L^1$ for which the operator $H_q:= -\Delta+q$ is nonnegative. In the final subsection we apply these properties to the study of the global well-posedness for the NV problem.

\subsection{ Inverting the Miura map}
Our main result here is the following:

\begin{theorem}\label{T-Minverse}
a) The Miura map is injective from $L^2$ to $\dot H^{-1} + L^1$.

b) The inverse of the Miura map is continuous on $\calM(L^2)$.
    
\end{theorem}

\

\begin{proof}

We begin with some general considerations. Let $u \in L^2$  satisfying \eqref{u-constraint.bis},
and $q = \calM(u)$.
We consider the real valued function 
\[
\phi = 2 \dbar^{-1} u \in \dot H^1
\]
so that $\partial_1 \phi = \Re u$ and
$\partial_2 \phi = \Im u$. Defining the linear operators
\[
 L_j := \partial_j - \partial_j \phi,
\]
it is then easy to see that the operator $H_q$ can be represented as
\begin{equation}\label{Hq-rep}
H_q = L_1^* L_1 + L_2^* L_2,    
\end{equation}
and the corresponding quadratic form
can be expressed as
\begin{equation}\label{Hq-rep-qf}
\langle H_q v, v\rangle= \|L_1 v\|_{L^2}^2 + \|L_2 v\|_{L^2}^2.    
\end{equation}
We remark that, with $L_u$ defined in \eqref{Lu}, we have 
\[
L_u = L_1 + i L_2.
\]
With this set-up in place, we turn our attention to the proof of the theorem.

\bigskip

a) Consider $u^1, u^2 \in L^2$ with $\partial u^1, \partial u^2$
real so that $\calM(u_1) = \calM(u_2) = q$. We want to show that 
$u_1= u_2$.

We start from the two different representations of $H_q$,
which at the level of bilinear forms shows that 
\[
\| L_{u_1} v \|_{L^2}^2 = \| L_{u_2} v\|_{L^2}^2
\]
for every real valued test function $v$, or, by density, for 
every $v \in \dot H^1 \cap L^\infty$. Writing 
\[
u^1 = \dbar \psi^1, \qquad u^2 = \dbar \psi^2, \qquad \psi^1,\psi^2 \in \dot H^1,
\]
the above identity becomes 
\[
\| (\partial_1  - \partial_1 \psi^1) v \|_{L^2}^2
+ \| (\partial_2  - \partial_2 \psi^1) v \|_{L^2}^2 = 
\| (\partial_1  - \partial_1 \psi^2) v \|_{L^2}^2
+ \| (\partial_2  - \partial_2 \psi^2) v \|_{L^2}^2
\]
If we knew say that $\psi_1$ is a nice (e.g. Schwartz) function then  this could be simplified by conjugating $w = e^{\psi_1} v$. This reduces the problem to the case when $\psi^1= 0$ with $\psi^2$  replaced by $ \psi^2-\psi^1$. So suppose that 
\[
\| \partial_1   v \|_{L^2}^2
+ \| \partial_2  v \|_{L^2}^2 = 
\| (\partial_1  - \partial_1 (\psi^2-\psi^1)) v \|_{L^2}^2
+ \| (\partial_2  - \partial_2 (\psi^2-\psi^1) v \|_{L^2}^2
\]
Then it suffices to consider a sequence $v_n \in \dot H^1$
so that 

(i) $v_n$ is uniformly bounded

(ii) $v_n \to 0$ in $\dot H^1$

(iii) $v_n \to 1$ a.e.

Such a sequence is easily seen to exist. Then passing to the limit in the previous identity we get
\[
0 = \| \partial_1 (\psi^2-\psi^1)\|_{L^2}^2+ \| \partial_2 (\psi^2-\psi^1)\|_{L^2}^2
\]
from which we conclude that $\psi^2 -\psi^1$ is constant therefore $u^1= u^2$.

\medskip 

What if $\psi^1$ is not a nice function ? Then we can still emulate the above argument. For any nice function $\psi$ we conjugate  $w = e^{\psi} v$
to arrive at 
\begin{multline*}
\| (\partial_1  - \partial_1 (\psi^1-\psi) v \|_{L^2}^2
+ \| (\partial_2  - \partial_2 (\psi^1-\psi)) v \|_{L^2}^2\\  = 
\| (\partial_1  - \partial_1 (\psi^2-\psi)) v \|_{L^2}^2
+ \| (\partial_2  - \partial_2 (\psi^2-\psi)) v \|_{L^2}^2
\end{multline*}
Given any $\epsilon > 0$, we can choose $\psi$ so that 
\begin{equation}
\| \psi^1 - \psi\|_{\dot H^1} \leq \epsilon. 
\end{equation}
Then we have
\[
\| (\partial_1  - \partial_1 (\psi^2-\psi)) v \|_{L^2}^2
+ \| (\partial_2  - \partial_2 (\psi^2-\psi)) v \|_{L^2}^2
\lesssim \| v\|_{\dot H^1}^2 + \epsilon \| v\|_{L^\infty}
\]
Using the same sequence $v_n$ as above, in the limit it follows 
that 
\begin{equation}
  \| \psi^2 - \psi\|_{\dot H^1} \lesssim \epsilon.   
\end{equation}
We conclude that 
\begin{equation}
  \| \psi^2 - \psi^1\|_{\dot H^1} \lesssim \epsilon.   
\end{equation}
But $\epsilon$ was arbitrary, so $\psi_1 = \psi_2$.

b) We use the same idea as in part (a). Suppose we have 
a sequence $u^n$ and $u$ and corresponding $q^n$ and $q$  (and also $\psi^n$ and $\psi$)
so that 
\begin{equation}
  q_n \to q \qquad \text{in }\dot H^{-1}+L^1.  
\end{equation}
Then for each $v \in \dot H^1\cap L^\infty$ we have 
\[
\lim_{n \to \infty} \| (\partial_1  - \partial_1 \psi^n) v \|_{L^2}^2
+ \| (\partial_2  - \partial_2 \psi^n) v \|_{L^2}^2 = 
\| (\partial_1  - \partial_1 \psi) v \|_{L^2}^2
+ \| (\partial_2  - \partial_2 \psi) v \|_{L^2}^2
\]
uniformly for $v$ in a bounded set.

If $\psi$ is a nice function then we conjugate by $e^{\psi}$
to arrive at 
\[
\lim_{n \to \infty} \| (\partial_1  - \partial_1 (\psi^n-\psi)) v \|_{L^2}^2
+ \| (\partial_2  - \partial_2 (\psi^n-\psi)) v \|_{L^2}^2 = 
\| \partial_1 v \|_{L^2}^2
+ \| \partial_2  v \|_{L^2}^2
\]
still uniformly for $v$ in a bounded set. Now we use the same sequence $v_m$ as above. Given $n$, we can choose $m_n > n$
so that 
\[
\| (\partial(\psi^n-\psi) v_{m_n}\|_{L^2} \lesssim \frac{1}{n}
\]
Then we set $v = v_{m_n}$ above and let $n \to \infty$ (which we can do because of the uniformity in $v$)  to obtain
\[
\lim_{n \to \infty}  \|\partial (\psi^n - \psi)\|_{L^2} = 0
\]
as needed.

Suppose now that $\psi$ is not a nice function. Then for each $\epsilon > 0$ we find $\psi_\epsilon$ which is nice 
and so that 
\[
\| \psi^\epsilon - \psi\|_{\dot H^1} \leq \epsilon. 
\]
Then we conjugate by $\psi_\epsilon$ to obtain 
\begin{multline*}
\lim_{n \to \infty} \| (\partial_1  - \partial_1 (\psi^n-\psi_\epsilon)) v \|_{L^2}^2
+ \| (\partial_2  - \partial_2 (\psi^n-\psi_\epsilon)) v \|_{L^2}^2 = \\ 
\| (\partial_1  - \partial_1 (\psi-\psi_\epsilon)) v \|_{L^2}^2
+ \| (\partial_2  - \partial_2 (\psi-\psi_\epsilon)) v \|_{L^2}^2
\end{multline*}
This is still uniformly in $v$ for fixed $\epsilon$, but not uniformly in $\epsilon$.

Inserting here $v= v_{n_{m}}$ as above (which also depends on $\epsilon$ but this is not important) we arrive at 
\[
\lim_{n \to \infty} \|\partial( \psi^n - \psi^\epsilon) \|_{L^2}
\lesssim \epsilon
\]
which implies 
\[
\limsup_{n \to \infty} \|\partial( \psi^n - \psi^\epsilon) \|_{L^2}
\lesssim \epsilon
\]
But $\epsilon$
is arbitrarily small, so we conclude that 
\[
\psi^n \to \psi \qquad \text{in }\dot H^{1},
\]
as needed.

{\bf Remark:} Apriori one expects $\|\psi_\epsilon\|_{L^\infty}$
to grow almost exponentially in $\epsilon$, so this does not imply Lipschitz (or even uniform) continuity of the inverse of the Miura map. This is as expected.
\end{proof}

\subsection{ The range of the Miura map and the Agmon--Allegretto--Piepenbrink (AAP) principle}\label{sec:spectral}

While the results in the previous subsection would suffice in order to complete a well-posedness theory for NV
in the range of the Miura map, a more complete result should include a characterization of this range.
The main goal of this subsection will be to address this question, and to connect the range of the Miura map with the 
spectral properties of the  Schr\"odinger operator
\begin{equation} \label{Lax}
H_q=-\Delta + q
\end{equation}
as an unbounded self-adjoint operator in $L^2(\R^2)$. This is expected since $H_q$ represents the Lax operator associated to the integrable NV flow. 
To get the discussion started,  is not difficult to see that for such potentials, the associated Schr\"odinger operator is nonnegative.
\begin{proposition}\label{p:positive}
Let $u \in L^2$ satisfying the constraint \eqref{u-constraint.bis},
and $q = \calM(u)$. Then $H_q$ is nonnegative.
\end{proposition}
\begin{proof}
If $q = \calM(u)$ then $H_q$ admits the decomposition 
\eqref{Hq-rep},
which immediately shows that the range of $\calM$ only contains potentials $q$ for which $H_q$ is nonnegative.
\end{proof}

The question is then whether the converse is true, which is a much more 
difficult question. This is where the Agmon--Allegretto--Piepenbrink (AAP) principle comes in, as it establishes an equivalence
between the spectral nonnegativity of the Schrödinger operator $H^q$ and the existence of a positive solution (or supersolution) of the associated elliptic equation. In a nutshell, the AAP principle states:
\medskip

\textbf{ The AAP principle in $\mathbb{R}^2$:}
\emph{The following are equivalent for real valued potentials $q$:
\begin{enumerate}
  \item The quadratic form associated with $H_q$ is nonnegative, i.e.,
  \[
    \int_{\mathbb{R}^2} \big( |\nabla v|^2 + q |v|^2 \big) \, dx \geq 0,
    \qquad \forall v \in C_c^\infty(\mathbb{R}^2).
  \]
  \item There exists a nontrivial positive weak solution $\psi>0$ of
  \[
    -\Delta \psi + q \psi = 0 \quad \text{in } \mathbb{R}^2.
  \]
\end{enumerate}}
\medskip
In dimension two, such a  characterization is particularly sharp because $-\Delta$
is critical with respect to Hardy-type inequalities. Such potentials are sometime called of \emph{conductivity type} due to the representation 
\[
q = \frac{\Delta \psi}{\psi},
\]
see, for example, the discussion in the introduction to \cite{MP2018}, and also the results below.

To formally close the circle and return to the Miura map, we remark that, for $\psi$ as above, the 
function $\phi = \log \psi$ solves
\begin{equation}\label{eq:phi}
\Delta \phi + |\nabla \phi|^2 = q
\end{equation}
Reinterpreting this  as an equation for $u = 2 \partial \phi$, we get nothing but
\[
\calM(u) = q,
\]
which heuristically justifies the converse to Proposition~\ref{p:positive}.

Whether the AAP principle is a theorem depends on  what is assumed about $q$.
This includes the case where $q \in L^p_u(\R^2)$ (the uniformly locally $L^p$ functions) for $p>1$ or more generally when $q$ is in the Kato class $K^2$. 
The AAP principle is closely related to work by Persson \cite{Persson1960} on the essential spectrum of elliptic operators and Agmon \cite{Agmon1979} on the exponential decay of solutions of elliptic equations $(P-\lambda)u=0$ for $\lambda$ below the essential spectrum.
 For the work of Allegretto and Piepenbrink in the 1970's and 1980's see Allegretto \cite{Allegretto74,Allegretto1979,Allegretto1981}, Piepenbrink \cite{Piepenbrink1974,Piepenbrink1977},  Allegretto-Piepenbrink, \cite{AllegrettoPiepenbrink1979}, and Moss-Piepenbrink \cite{MP1978}.
In \cite[\S C.8]{SimonSchrodinger}, Simon proved a version of the AAP principle for Schr\"{o}dinger operators in $L^2(\R^\nu)$ with potentials $q=q_+ +q_-$ for $q_+ \in K^\nu_{loc}$ and $q_- \in K^\nu$, where  $K^\nu_{loc}$ is the local Kato class and $q_+$ is the positive part of $q$ (see \cite[\S A2]{SimonSchrodinger} for a full discussion of $K^\nu$ spaces). For more recent work on the AAP principle, see for example \cite{BOP22} and references therein.

In all cases the regularity of $\psi$ is the object of Caccioppoli type inequalities, which have the form 
\begin{equation}
\int \chi^2 \frac{|\nabla \psi|^2}{\psi^2}\, dx \lesssim   \int |\nabla \chi|^2 \, dx + \int \chi^2 q \, dx  
\end{equation}
for any test function $\chi$,
 which in dimension two lead to the global bound 
\[
 \int \frac{|\nabla \psi|^2}{\psi^2} \leq \int q \, dx.
 \]
 More can be said if $q$ is  slightly better than $L^1_{loc}$. The case when $q \in L^p_{\mathrm{loc}}(\R^2)$ for some $p>1$, for instance, was considered in 
 \cite{Murata1986}.

Theorem 2.1.2 of \cite{CKFS1987} implies that, if $q \in L^p_{\mathrm{loc}}(\R^2)$ for some $p>1$, then $-\Delta +q \geq 0$ if and only if there is a positive distributional solution to $(-\Delta+q)\psi=0$. If $q$ has sufficient decay, $(-\Delta+q)\psi=0$ admits a positive distributional solution in the critical case, and a positive solution of logarithmic growth in the subcritical case (see \cite{CKFS1987}, Theorem 5.6). 
Let
$$ L^p_\rho(\R^2) = \{ f \in L^p(\R^2): (1+|x|^2)^{\rho/2} f(x) \in L^p(\R^2) \}. $$In 
\cite{MP2018} it is shown that, for $1<p<2$ and $\rho > 1 + 2/\tilde{p}$, the subcritical potentials are an open subset of $L^p_\rho(\R^2)$ and the critical potentials form the boundary of this set. In other words, in both cases $0$ is the bottom of the spectrum for $H_q$, but in the subcritical case $0$ is a regular point, while in the critical case $0$ is a resonance. Unfortunately, this line of inquiry does not go far enough such as to allow for potentials
$q \in \dot H^{-1} + L^1$.

 Another viewpoint, adopted for instance in \cite{JMV}, is to attempt to characterize all distributions $q$ so that $H_q$ is nonnegative. Restricted to dimension two, one of the results in  \cite{JMV} asserts that, if $H_q$
 is nonnegative then there exists $\Gamma \in L^2_{loc}$ so that 
 \[
q \geq \nabla \cdot \Gamma + \Gamma^2 
 \]
 This comes the closest to our requirements, but having inequality rather than equality above remains a major obstruction. This is necessary 
 e.g. for arbitrary measures $q$, 
 but we aim to remove it when $q \in \dot H^{-1} + L^1$.  

We now arrive at our main result, which provides a definitive answer to the questions above for real potentials $q \in \dot H^{-1} + L^1$.

\begin{theorem}\label{t:AAP}
The following are equivalent for  real potentials $q \in \dot H^{-1} + L^1$:

\begin{enumerate}[label=(\roman*)]
\item The operator $H_q$ is nonnegative. 

\item The equation $H_q \psi = 0$ admits a nonnegative solution $\psi$
with $\ln \psi \in \dot H^1$.

\item The potential $q$ can be represented as 
\[
q = \Delta \phi + |\nabla \phi|^2, \qquad \phi \in \dot H^1.
\]
\end{enumerate}

\end{theorem}

Our primary interest in this article 
is in the equivalence between (i) and (iii), but we have also added (ii) in order  to relate our result to the AAP principle.

\begin{remark}
    In one dimension, the Miura map $B(r) = r' +r^2$ is a map from $L^2_{loc}(\R)$ to $H^{-1}_{loc}(\R)$. Kappeler, Perry, Shubin, and Topolov \cite{KPST}, analyzed the one-dimensional Miura map and its relation to the Schr\"{o}dinger operator $H_q = -d^2/dx^2 + q$.  They showed that, for a real-valued distribution $q \in H^{-1}_{loc}$,  there is an  equivalence analogous to the one shown in Theorem \ref{t:AAP} (see \cite[Theorem 1.1]{KPST}).
\end{remark}

\begin{remark}
One delicate matter, in the case when $q \in \dot H^{-1} + L^1$, is the interpretation of the equation $H_q \psi = 0$ in (ii). The difficulty is that $\psi$ cannot be expected to be locally in $H^1 \cap L^\infty$, so the multiplication $q \psi$ is ill defined. We bypass this issue by reinterpreting the equation as
\begin{equation}\label{conduct}
\frac{\Delta \psi}{\psi} = q,
\end{equation}
where the left hand side can be properly and uniquely interpreted as
distribution whenever $\psi$ is a.e. positive with $\log \psi \in \dot H^1$.
On the other hand this difficulty dissapears when in addition $q \geq 0$,
in which case we do expect that $\psi$ is locally  in $H^1 \cap L^\infty$, essentially by a maximum principle type argument. This case plays an important role in our proof.
\end{remark}

 Returning to our Miura map, this result directly implies the converse to Proposition~\ref{p:positive}.
 
\begin{corollary}
Let $q \in   \dot H^{-1}+ L^1$. Then $H_q$ is nonnegative iff $q \in \calM(L^2)$.
\end{corollary}

It also has another, softer consequence:

\begin{corollary}
   $\calM(L^2)$ 
is closed in $\dot H^{-1} + L^1$. 
\end{corollary}

\begin{proof}[Proof of Theorem~\ref{t:AAP}]

\textbf{(ii) $\Leftrightarrow$  (iii)}
Denoting $\phi = \log \psi$, here it suffices to show that for $\phi \in \dot H^1$ we have the relation 
\[
\frac{\Delta e^\phi}{e^\phi} = \Delta \phi + |\nabla \phi|^2.
\]
But this is straightforward when $\phi$ is a smooth positive function.
The multiplication on the left hand side is also well-defined when $\phi \in  (H^1 \cap L^\infty)_{loc}$, in which case we have both $e^\phi$ and $e^{-\phi}$ in $H^1_{loc}$, and the above relation extends by density to this setting. Finally for arbitrary 
$\phi \in \dot H^1$ we note that the 
right hand side depends smoothly in $\phi$ in the $\dot H^{-1} + L^1$ topology, and we simply define 
the fraction on the left as the unique 
continuous extension of the right hand side, noting that $\dot H^1 \cap L^\infty$ is dense in $\dot H^1$.

\medskip

\textbf{(iii) $\Rightarrow$  (i)} This is the content of Proposition~\ref{p:positive}.

\medskip

\textbf{(i)  $\Rightarrow$ (iii)}
We consider $q \in \dot H^{-1} + L^1$ so that 
$H_q$ is nonnegative. We want to show that 
we have a representation 
\begin{equation}\label{q-rep}
q = \Delta \phi + |\nabla \phi|^2,
\end{equation}
with $\phi \in \dot H^1$.
The existence of such a $\phi$  in $\dot H^1_{loc}$ is known if $q$ is a nice function (e.g. $L^p_{loc}$) from classical versions of the AAP principle, see references above.

Then a Caccioppoli type inequality gives the desired $\dot H^1$ bound for $\phi$, as follows. Given a test function $\chi$, we test \eqref{q-rep} by $\chi^2$, integrate by parts and apply Cauchy-Schwarz. This yields
\begin{equation}\label{caccioppoli-pf}
 \int \chi^2 |\nabla \phi|^2 dx \lesssim \int |\nabla \chi|^2  dx + \|q\|_{\dot H^{-1} + L^1} \| \chi\|_{\dot H^1 \cap L^\infty}.   
\end{equation}
To get the $\dot H^1$ bound for $\phi$ it suffices to consider a sequence of nonnegative test functions $\chi_R$ which as $R \to \infty$ satisfy the following properties:
\begin{enumerate}[label=(\alph*)]
\item 
\[
0 \leq \chi_R \leq 1, \qquad \chi_R = 0 \quad |x| > R
\]
\item 
\[
\chi_R \to 1 \quad a.e. 
\]
\item 
\[
\lim_{R \to \infty} \|\chi_R\|_{\dot H^1}= 0
\]
\end{enumerate}
Such a sequence is easily seen to exist. Boundedness rather than decay would suffice in (c), but the better choice will be useful later. With this choice, we arrive at
\begin{equation}\label{caccioppoli-pf1}
 \int|\nabla \phi|^2 dx \lesssim  \|q\|_{\dot H^{-1} + L^1}.
\end{equation}

The transition from nice $q$ to $q \in \dot H^{-1} + L^1$
is accomplished in two steps:
\bigskip

{\bf A. $q$ bounded non-negative Radon measure. }
Here (i) automatically holds, and our objective will be to prove 
(iii). We divide the argument into 
several modules:

\bigskip

\emph{ A1: $H_q$ as a coercive operator.}
To begin with, we consider the operator $H_q$ as a nonnegative operator
on test functions, and seek to extend its associated bilinear form 
by density to its domain 
\[
D_q \subset \dot H^1 \cap L^2(q)
\]
The delicate part in this is whether 
$H^1$ functions are even defined $q$- a.e. This is false for some nonnegative measures, e.g. any Dirac mass. This is clarified using the notion of $H^1$ capacity of a set. 

The above mentioned extension is possible  if and only if for every Borel set $E$ we have 
\begin{equation} \label{ext-thm}
\mathrm{Cap}_{H^1}(E) = 0 \implies q(E) = 0,
\end{equation}
see \cite[Chapter 2]{Mazya2011}.
Here the $\dot H^1$ capacity of a compact set is defined, see e.g. \cite{Adams-Hedberg}, as 
\[
\mathrm{Cap}_{H^1}(E)= \inf\{
\| \phi \|_{H^1}\ |\  \phi \in S,
\ \phi \geq 1 \text  { on } E \}   
\]
In particular, for measures $q$ satisfying \eqref{ext-thm}, $\dot H^1$ functions can be uniquely defined as continuous $q$-a.e.
(called the quasicontinuous extension) see \cite[Proposition 6.1.2]{Adams-Hedberg}. One can also show that $D_q$ contains $\dot H^1 \cap L^\infty$, as classical regularizations of such functions will converge pointwise $q$-a.e. to their 
quasicontinuous extension, and thus in $L^2(q)$ by dominated convergence.

While the above property \eqref{ext-thm} does not hold for arbitrary measures,  we will prove that this is indeed the case if $q \in L^1+ \dot H^1$. 
We consider a representation 
\[
q = q_1+q_2, \qquad q_1 \in L^1, \qquad q_2 \in \dot H^{-1}. 
\]
This is a local property, so without loss of generality we assume that  $q$ is in the space $L^1 + H^{-1}$  with both components localized inside the unit ball $B$, and $E \subset B$; this can be easily achieved using appropriate cutoff functions.

Let $E \subset B$ be a Borel set with $\mathrm{Cap}_{H^1}(E)=0$;
in particular it must also have Lebesque measure zero. We use an equivalent characterization of zero capacity sets, which is that for each $\epsilon > 0$ 
there exists a nonnegative function $\chi \in H^1_0(2B)$ with the following properties:

\begin{enumerate}
    \item $\chi\leq 1$ everywhere and $\chi = 1$ in a neighbourhood of $E$
    \item $\chi$ is small in $H^1$, $\| \chi\|_{H^1} \leq \epsilon$.
\end{enumerate}
These properties imply in particular that $\chi$ is pointwise small outside a small set,
\begin{equation}\label{small-set}
m(\{ \chi > \epsilon^\frac12 \}) < \epsilon
\end{equation}

We use $\chi$ to compute 
\[
q(E) \leq \int \chi dq = 
\int \chi q_1 dx + \langle \chi, q_2 \rangle  \lesssim \epsilon + \int_{\{ \chi > \epsilon^\frac12 \}} |q_1|dx
\]
Letting $\epsilon \to 0$ on the right 
and using \eqref{small-set} we obtain $q(E) = 0$ as needed.

\bigskip
\emph{ A2: Approximate zero modes for $H_q$.}
We begin our construction of  a nonnegative zero mode for $H_q$ by successively solving a sequence of elliptic problems
\[
H_q \psi_R = c_R \chi_R
\]
where for each $R$, $\chi_R$ is a nonnegative bump function supported 
in $B_{2R} \setminus B_R$. Such solutions exist in $D_q$ and are unique in by Lax-Milgram, and are also nonnegative in view of the 
standard variational characterization.
We normalize the constant $c_R$ so that 
\begin{equation} \label{psi-unif}
\|\psi_R\|_{L^2(B)} = 1.
\end{equation}
Here without loss of generality we choose the unit ball $B$ so that $q(B) > 0$, which in turn guarantees that we have the Poincare type inequality 
\begin{equation}
\int_B \psi^2 dx \lesssim \int_B |\nabla \psi|^2 dx + \int_B \psi^2 dq 
\end{equation} 
We remark that the functions $\psi_R \in H^1_{loc}$ are subharmonic and nonnegative in $B_R$, so they are also bounded say in $B_{R/2}$.

To summarize, for each $R$
we have obtained nonnegative 
functions $\psi_R \in H^1 \cap L^\infty(B_{R/2})$ which solve 
$H_q \psi_R = 0$ in the same ball,
with the normalization \eqref{psi-unif}. This is all the information we carry forward. 

\bigskip

\emph{A3: Uniform bounds for nonnegative local zero modes.}
In order to obtain a global zero mode $\psi$ for $H_q$ as the local limit 
of $\psi_R$ on a subsequence, we need 
bounds for $\psi_R$ which are uniform in $R$. For clarity we drop the $R$ subscript and simply assume that we have a nonnegative solution $\psi$ to $H_q \psi = 0$ in a ball $B_R$, with the normalization \eqref{psi-unif}.
 We do this using again Caccioppoli 
type inequalities.

Ideally we would like to test the equation with a function of the form 
$\dfrac{\chi^2}{\psi}$ where the test function $\chi$ is for now arbitrary. But we do not know that $\psi$ stays away from zero,  so we penalize this 
and replace it by $\dfrac{\chi^2}{\psi+\epsilon}$ with a small parameter $\epsilon$. This belongs to $H^1 \cap L^\infty \subset D_q$. We obtain
\begin{equation}\label{caccipoli}
\begin{aligned}
0 = & \ \int \nabla \psi \cdot \nabla  \frac{\chi^2}{\psi+\epsilon}\, dx + 
\int \frac{ \psi \chi^2}{\psi+\epsilon}dq 
 \\ 
 =&\ - \int   \frac{|\nabla \psi|^2\chi^2}{(\psi+\epsilon)^2}\, dx +
\int 2\chi \frac{\nabla \psi \cdot \nabla \chi}{\psi+\epsilon}  
+ \int \frac{ \psi \chi^2}{\psi+\epsilon}dq 
\end{aligned}
\end{equation}
By Cauchy-Schwartz this yields
\[
\int   \frac{|\nabla \psi|^2\chi^2}{(\psi+\epsilon)^2}\, dx \lesssim 
\int |\nabla \chi|^2 \, dx + \int \chi^2\,  dq
\]
Choosing  $\chi$ to be a bump function
adapted to $B_r$, the right hand side stays bounded and we arrive at 
\[
\| \log (\psi+\epsilon)\|_{\dot H^1(B_r)} \lesssim 1
\]
The normalization \eqref{psi-unif}
allows us to bound also the $L^2$ norm, 
first in $B$,
\[
\| \log (\psi+\epsilon)\|_{L^2(B)} \lesssim 1
\]
and then by Hardy in $B_r$, 
\[
\| \log (\psi+\epsilon)\|_{L^2(B_r)} \lesssim r
\]
This suffices in order to pass to the limit $\epsilon \to 0$ weakly in $H^1$
and strongly in $L^2$ sense and obtain 
$\log \psi \in H^1(B_r)$, with 
the  universal bound  
\begin{equation}\label{caccipoli+}
\| \log \psi\|_{\dot H^1(B_r)} \lesssim 1.
\end{equation}
and in particular $\psi > 0$ a.e.
A more subtle point here is that, since $q$ vanishes on capacity zero sets,
we also have $\psi > 0$ $q$-a.e.

This in turn implies that we have 
strong convergence, 
\[
\log (\psi+\epsilon) \to \log \psi \qquad \text{in}  \quad H^1(B_r).
\]
From here, by Trudinger-Moser type inequalities it follows that 
\begin{equation}\label{estr1}
 \| \psi\|_{L^2(B_r)} \lesssim_r 1
\end{equation}
Then, by localized energy estimates 
we get 
\begin{equation}\label{estr2}
 \| \nabla \psi\|_{L^2(B_{r/2})} \lesssim_r 1
\end{equation}
Finally, since $\psi$ is nonnegative and subharmonic in $B_r$, it also follows that we have a local uniform bound
\begin{equation}\label{estr3}
\| \psi\|_{L^\infty(B_{r/2})} \lesssim_r 1.
\end{equation}

\bigskip

\emph{A4: The nonnegative zero mode $\psi$ for $H_q$.}
 We seek to obtain $\psi$ as the local limit, on a subsequence, of the functions $\psi_R$ as $R\to \infty$.
 The estimates in the previous step show that the bounds \eqref{estr1}, \eqref{estr2} and \eqref{estr3} hold uniformly for $\psi_R$ if $r \ll R$.

 Then we consider the localized functions $\chi_{<r} \psi_R$ which solve 
\[
H_q (\chi_{<r} \psi_R) = -2 \nabla \chi_{<r} \nabla \psi_R - \Delta \chi_{<r} \psi_R. 
\]
As $R \to \infty$, the right hand sides are uniformly bounded in $L^2$
so they converge on a subsequence in $H^{-1}$. We conclude that the functions $\chi_{<r} \psi_R$ converge on a subsequence in $D_q$. Applying this for a sequence of scales $r_n \to \infty$, we obtain in the limit 
a nonnegative function $(\psi \in H^1 \cap L^\infty)_{loc} \subset D_{q,loc}$ which solves $H_q \psi = 0$
and $\|\psi\|_{L^2(B)} = 1$.

The earlier energy estimates in \eqref{caccipoli+} apply to our zero mode for every $r$, so we obtain the global bound 
\begin{equation}
\| \nabla \log \psi\|_{L^2} \lesssim 1. \end{equation}

In particular the identity 
in \eqref{caccipoli} can be applied with a sequence of cutoffs $\chi_R$ as discussed earlier, which converges to $1$ pointwise and to $0$ in $\dot H^1$.
In the limit as $R \to \infty$ we obtain the identity 
\[
\int |\nabla \log (\psi+\epsilon)|^2 \, dx = \int \frac{\psi}{\psi+\epsilon}\, dq
\]
which by dominated convergence yields, as $\epsilon \to 0$,
\begin{equation}
\int |\nabla \log (\psi)|^2 \, dx = \int 1 \, dq.
\end{equation}
We note that in particular this is part of property  (ii).

\bigskip

\emph{The proof of (iii).}
Our objective here is to show that the 
function $\phi = \log \psi$ solves \eqref{eq:phi}. This is obvious formally, but the direct computation
is not rigorous at this low regularity 
since $\psi$ is not guaranteed to stay away from zero locally. So instead we compute an equation for $\log(\psi+\epsilon)$, which is locally bounded and in $H^1$. Then we have
\[
\Delta \log (\psi+\epsilon)
= \frac{\Delta \psi}{\psi+\epsilon} 
-\frac{|\nabla \psi|^2}{(\psi+\epsilon)^2} = \frac{\psi}{\psi+\epsilon} q - |\nabla \log(\psi+\epsilon)|^2
\]
qr equivalently
\[
\Delta \log (\psi+\epsilon) + 
|\nabla \log(\psi+\epsilon)|^2 = \frac{\psi}{\psi+\epsilon} q 
\]
As proved earlier, $\log (\psi+\epsilon)
$ converges to $\log \psi$ locally in $H^1$, so we can pass to the limit on the left in the sense of distributions.

On the other hand on the right, we can use dominated convergence, taking advantage ofthe fact, established earlier, that $\psi > 0$ $q$-a.e.
Hence \eqref{eq:phi} follows, which concludes the proof of (iii) for nonnegative $q$.

\bigskip 

{\bf B. General $q$ in $ \dot H^{-1} + L^1$.}
For $q$ which is merely in $\dot H^{-1} + L^1$ we consider its regularizations $q_\epsilon$ on the $\epsilon$ scale. By averaging the operators $H_{q_{\epsilon}}$ are still nonnegative  
(see e.g. Step 1 of the proof  of \cite[Theorem 5.1]{JMV}).
The mollified functions $q_\epsilon$ are also smooth, and thus 
admit smooth representations as in \eqref{q-rep},
\[
q_\epsilon = \Delta \phi_\epsilon + |\nabla \phi_\epsilon|^2.
\]
By the Cacciopolli inequality above, the functions $\phi_\epsilon$  are uniformly bounded in $\dot H^1$. Using this property and testing again with $\chi_R^2$
we obtain, as $R \to \infty$, the identity
\[
\| \phi_\epsilon \|_{\dot H^{1}}^2 
= \int q_\epsilon
\lesssim 
\|q_\epsilon \|_{\dot H^{-1} + L^1}
\]
At this point, the argument in \cite{JMV}
continues by taking a weak limit on a subsequence (after normalizing constants). 
Instead, we will show that on a subsequence we have a strong limit
\[
\psi = \lim \psi_\epsilon \qquad in \qquad \dot H^1
\]
If that holds then the proof is concluded. For a direct application of the Riesz-Kolmogorov theorem, one needs two properties:
\begin{enumerate}[label=(\Alph*)]
\item Tightness at infinity,
\begin{equation}
\lim_{R \to \infty} \sup \| \nabla \phi_\epsilon \|_{L^2(|x|> R)} = 0     
\end{equation}
\item Equicontinuity,
\begin{equation}
\lim_{\lambda  \to \infty} \sup \| P_{> \lambda} \nabla \phi_\epsilon \|_{L^2} = 0     
\end{equation}
\end{enumerate}

We can establish part A by using  test functions $\chi_R$ as above. We then have 
\[
\int (1-\chi_R) |\nabla \psi_\epsilon|^2 dx
= - \int \nabla \chi_r \cdot \nabla \psi_\epsilon
+ \int (1-\chi_R) q_\epsilon 
\]
where the right hand side must go to zero as $R \to \infty$, uniformly in $\epsilon$.

However part B is more delicate. We argue by contradiction. 
If equicontinuity does not hold, then on a subsequence we must have a low-high decomposition,
\[
\psi_{\epsilon} = \psi_\epsilon^{lo} + \psi_\epsilon^{hi} 
\]
so that the first component converges in $\dot H^1$
\[
\lim_{\epsilon \to 0} \psi_\epsilon^{lo} = \phi^{lo}
\]
while the second component runs to infinity in frequency, with
\begin{equation}\label{to-infty0}
\lim_{\epsilon \to 0} \| \phi_\epsilon^{hi}\|_{\dot H^1} > 0, 
\end{equation}
while
\begin{equation}\label{to-infty}
\lim_{\epsilon \to 0} P_{< \lambda} \phi_\epsilon^{hi} = 0, \qquad \lambda > 0
\end{equation}
Then it is easily verified that the interaction term decays,
\[
\| \nabla \phi_\epsilon^{lo} \nabla \phi_\epsilon^{hi}\|_{L^1+\dot H^{-1}} = 0
\]
so we are left with 
\[
\lim_{\epsilon \to 0} \Delta \phi_\epsilon^{hi} + |\nabla \phi_\epsilon^{hi}|^2 = q^{hi} : = q - (\Delta \phi^{lo} + |\nabla \phi^{lo}|^2) 
\]
Here $\Delta \phi_\epsilon^{hi}$ converges to $0$ in the sense of distributions, so we conclude that in the  sense of distributions we have
$q^{hi} =  \lim_{\epsilon \to 0} |\nabla \phi_\epsilon^{hi}|^2 \geq 0$. In other words, $q^{hi} \in L^1+\dot H^1$ is a nonnegative Radon measure.

But from  part A of the proof we know that for such $q^{hi}$ we do have a representation 
\[
q^{hi} = \Delta \phi^{hi} + |\nabla \phi^{hi}|^2
\]
By the continuity of the inverse of the Miura map in Theorem~\ref{T-Minverse},
it then follows that $\phi_\epsilon^{hi} \to \phi^{hi}$
in $\dot H^1$. This contradicts \eqref{to-infty0}-\eqref{to-infty}, and completes the proof of this step.
\end{proof}

\subsection{\texorpdfstring{$\dot H^{-1} + L^1$ solutions for NV}{H-1 + L1 solutions for NV}}
Finally, we turn our attention to Theorem~\ref{thm:NV}.

\begin{proof}[Proof of Theorem~\ref{thm:NV}]
For any NV initial data $q_0 \in \calM(L^2)$ we define the associated data $u_0$ to mNV by inverting $\calM$,
\[
u_0 = \calM^{-1} q_0.
\]
Let $u$ be the global mNV solution with this initial data.
Then we define the solution $q$ to NV to be 
\[
q(t) = \calM u(t).
\]
By the continuity properties of $\calM$ and its inverse
it follows that $q \in C(\R,\dot H^{-1}+ L^1)$ and in addition the data to solution map is continuous in the same topology.

We next show that this solution operator is the unique continuous extension of the classical solution operator for more regular data. For this, we need to verify that $L^2 \cap \calM(L^2)$
is dense in $\calM(L^2)$.  This is easily achieved by regularizing $u_0$ instead, approximating it with $H^1$ functions, whose Miura map then belongs to $L^2$.

Lastly, we show that the solutions $q$ to NV satisfy the Strichartz bounds \eqref{NV-Str}. For this we use in the Miura map  the corresponding Strichartz bounds for mNV, which hold by Theorem~\ref{thm:mNV}. We obviously have
\[
\| \partial u\|_{DS^p} \lesssim \|u\|_{S^p},
\]
so it remains to consider the quadratic term, where 
it suffices to estimate
\begin{equation}\label{DSp-u2}
 \| |u|^2\|_{DS^p} \lesssim \|u\|_{\ell^2 S^{2p}}^2.   
\end{equation}
Here the pairs of Strichartz exponents are denoted by $(p,q)$
respectively $(2p,q_1)$ where it is easily seen that 
\begin{equation}\label{qq1}
\frac{2}{q_1}= \frac{1}{q}+ \frac12.
\end{equation}
To prove \eqref{DSp-u2} we consider the standard Littlewood-Paley trichotomy, writing
\[
|u|^2 = \sum_{2\mu < \lambda} u_\mu \bu_\lambda + u_\lambda \bu_\mu 
+  \sum_{\mu \approx \lambda} u_\mu \bu_\lambda: = q_1 + q_2
\]

To bound $q_1$ we note that the summands are frequency localized 
at frequency $\lambda$, therefore using standard square function bounds and Bernstein's inequality we have
\[
\begin{aligned}
\| q_1\|_{DS^p}^2 \lesssim & \ \|q_1\|_{\ell^2 DS^p}^2
\\
\lesssim & \ \sum_\lambda \lambda^{-2+\frac{1}p} \| \sum_{2\mu < \lambda}  u_\mu \bu_\lambda + u_\lambda \bu_\mu\|_{L^p L^q}^2
\\
\lesssim & \ \sum_\lambda \lambda^{-2+\frac{1}p} 
\|u_\lambda \|_{L^{2p} L^{q}}^2
( \sum_{2\mu < \lambda}  \|u_\mu \|_{L^{2p} L^\infty})^2
\\
\lesssim & \ \sum_\lambda \lambda^{-2+\frac{1}{p} + 2(\frac{1}{q_1}-\frac{1}q)} 
\|u_\lambda \|_{L^{2p} L^{q_1}}^2
( \sum_{2\mu < \lambda} \mu^{ \frac{2}{q_1}}  \|u_\mu \|_{L^{2p} L^{q_1}})^2
\\
\lesssim & \ \sum_\lambda \lambda^{-2+\frac{1}{2p} + 2(\frac{1}{q_1}-\frac{1}q)} 
\|u_\lambda \|_{S^{2p}}^2
( \sum_{2\mu < \lambda} \mu^{-\frac{1}{2p} + \frac{2}{q_1}}  \|u_\mu \|_{S^{2p}})^2
\end{aligned}
\]
In the last expression the sum of the exponents of $\lambda$ and $\mu$ vanishes and the exponent of $\mu$ is positive, so we arrive at 
\[
\| q_1\|_{DS^p}^2 \lesssim \sum_\lambda \|u_\lambda \|_{S^{2p}}^2
\| u\|_{\ell^2 S^{2p}}^2 \lesssim \| u\|_{\ell^2 S^{2p}}^4, 
\]
as needed.

To bound $q_1$ we note that the summands are frequency localized 
at frequency $\lesssim \lambda$, therefore using  Bernstein's inequality we have
\[
\begin{aligned}
\| q_2\|_{DS^p} 
\lesssim & \ \sum_{\mu \approx \lambda} \| |D|^{\frac{1}{p}-1} (u_\mu \bu_\lambda)\|_{L^p L^{q}}
\\
\lesssim & \ \sum_{\mu \approx \lambda} \| |D|^{\frac{1}{p}} (u_\mu \bu_\lambda)\|_{L^p L^{\frac{q_1}2}}
\\
\lesssim & \ \sum_{\mu \approx \lambda} \mu^{\frac{1}{2p}} \|u_\mu\|_{L^{2p}L^{q_1}}\lambda^{\frac{1}{2p}} \|u_\lambda\|_{L^{2p}L^{q_1}}
\\
\lesssim & \| u\|_{S^{2p}}^2
\end{aligned}
\]

\end{proof}

\bibliographystyle{amsplain}
\bibliography{mNV.bib}	
	
\end{document}